\documentclass[11pt]{gsm-l}

\usepackage{extsizes}

\usepackage{tikz}
\usepackage{rotating}

\usepackage{enumitem}
\setlist[itemize]{topsep=0pt,itemsep=0pt}
\setlist[enumerate]{topsep=0pt,itemsep=0pt}

\usepackage{scalerel}
\usepackage{tkz-euclide}
\usetikzlibrary{arrows, arrows.meta, bending, decorations.pathmorphing, decorations.markings} % hobby} }

\usepackage{auto-pst-pdf}

\usepackage{pict2e, amssymb,latexsym,amstext,amsmath,amsthm,epsfig,ifthen}
\usepackage{color}
\usepackage{blkarray}
\usepackage[capitalize, nameinlink, noabbrev, nosort]{cleveref}
\usepackage{youngtab}
\usepackage[boxsize=.35 em]{ytableau}

\usepackage[dvips,centering,includehead,width=12.61cm,height=22.1cm,paperheight=23.8cm,paperwidth=17cm]{geometry}

\input epsf

\makeatletter
\newcommand{\manuallabel}[2]{\def\@currentlabel{#2}\label{#1}}
\makeatother

\manuallabel{ex:SL_2}{1.1.2}
\manuallabel{sec:Ptolemy}{1.2}
\manuallabel{eq:grassmann-plucker-3term}{1.2.1}
\manuallabel{sec:baseaffine}{1.3}
\manuallabel{prop:Muir}{1.3.5}
\manuallabel{fig:moves}{1.10}
\manuallabel{fig:moves1}{1.7}

\manuallabel{def:Q(T)-polygon}{2.2.1}
\manuallabel{def:quiverwd}{2.3.1}
\manuallabel{sec:mut-double-wiring}{2.4}
\manuallabel{fig:markov-quiver}{2.10}
\manuallabel{eq:exch-rel-geom}{3.1.1}
\manuallabel{exercise:quiverexchange}{3.1.3}
\manuallabel{def:Tn}{3.1.4}
\manuallabel{def:cluster-algebra}{3.1.6}
\manuallabel{ex:base-affine-SL3}{3.2.1}
\manuallabel{example:A(1,2)}{3.2.7}
\manuallabel{thm:Laurent}{3.3.1}
\manuallabel{th:Laurent-sharper}{3.3.6}

\manuallabel{ex:mutation-acyclic}{4.1.5}
\manuallabel{def:clustersubalgebra}{4.2.6}

\manuallabel{th:finite-type-classification}{5.2.8}
\manuallabel{sec:type-A}{5.3}
\manuallabel{example:SL4}{5.3.8}
\manuallabel{sec:type-D}{5.4}
\manuallabel{def:arcs-Dn}{5.4.3}
\manuallabel{def:P-gamma}{5.4.9}
\manuallabel{eq:Omega}{5.3.2}
\manuallabel{eq:Adefn}{5.4.1}
\manuallabel{tab:cluster-numbers}{5.17}

\manuallabel{ch:plabic}{7}
\manuallabel{ch:Grassmannians}{8}

\newtheorem{theorem}{Theorem}[section]
\newtheorem{proposition}[theorem]{Proposition}

\newtheorem{corollary}[theorem]{Corollary}
\newtheorem{lemma}[theorem]{Lemma}

\theoremstyle{definition}

\newtheorem{remark}[theorem]{Remark}
\newtheorem{example}[theorem]{Example}
\newtheorem{definition}[theorem]{Definition}

\newtheorem{exercise}[theorem]{Exercise}

\numberwithin{section}{chapter}
\numberwithin{equation}{section}
\numberwithin{figure}{chapter}
\numberwithin{table}{chapter}

\def\endproof{\hfill$\square$\medskip}

\def\ZZ{\mathbb{Z}}

\def\CC{\mathbb{C}}

\def\TT{\mathbb{T}}

\def\AA{\mathcal{A}}
\def\FFcal{\mathcal{F}}

\def\xx{\mathbf{x}}

\def\zz{\mathbf{z}}
\def\Rcal{\mathcal{R}}
\def\Xcal{\mathcal{X}}

\def\Gr{\operatorname{Gr}}

\def\SL{\operatorname{SL}}
\def\GL{\operatorname{GL}}

\newcommand{\light}[1]{{#1}}
\newcommand{\dark}[1]{{#1}}

\newcommand{\matbr}[4]{\left[\!\!\begin{array}{cc}
#1 & #2 \\
#3 & #4 \\
\end{array}\!\!\right]}

\newcommand{\CGr}[2]{\CC[\widehat{\Gr}_{#1,#2}]}

\newcommand{\overunder}[2]{
\!\begin{array}{c}
\scriptstyle{#1}\\[-.1in]
-\!\!\!-\!\!\!-\\[-.1in]
\scriptstyle{#2}
\end{array}
\!
}

\newcommand{\notch}{\scriptstyle\bowtie}

\newcommand{\bmat}[4]{\left[\!\!\begin{array}{cc} 
#1 & #2 \\ #3 & #4 \\ 
\end{array}\!\!\right]}

\newsavebox{\digon}
\savebox{\digon}(10,10)[bl]{
\qbezier(5,0)(0,5)(5,10)
\qbezier(5,0)(10,5)(5,10)
\put(5,0){\circle*{1}} 
\put(5,5){\circle*{1}} 
\put(5,10){\circle*{1}} 
}

\newsavebox{\lowbar}
\savebox{\lowbar}(10,10)[bl]{
\put(5,0){\line(0,1){5}}
}

\newsavebox{\lowtag}
\savebox{\lowtag}(10,10)[bl]{
\put(5,3.5){\makebox(0,0){$\notch$}}
}

\newsavebox{\highbar}
\savebox{\highbar}(10,10)[bl]{
\put(5,5){\line(0,1){5}}
}

\newsavebox{\hightag}
\savebox{\hightag}(10,10)[bl]{
\put(5,6.5){\makebox(0,0){$\notch$}}
}

\definecolor{darkred}{rgb}{1,0,0}        
\definecolor{lightred}{rgb}{1,0.4,0}     
\definecolor{darkblue}{cmyk}{1,0.4,0,0.4}  %blue
\definecolor{lightblue}{cmyk}{1,0.4,0,0}  %blue
\definecolor{darkgreen}{cmyk}{1,0.5,1,0}  
\definecolor{lightgreen}{cmyk}{1,0,1,0}  
\makeindex

\begin{document}

%%%%%%%%%%%%% mock manual labels temporarily created to verify the actual ones
%%%%%%%%%%%%% comment out all chapters in \includeonly above, 
%%%%%%%%%%%%% then run the commands below for all labels to see their values 

%\makeatletter
%\newcommand{\mockmanuallabel}[2]{\ref{#1}\ \ #2\ \ \def\@currentlabel{#2}\label{#1}\par}
%\makeatother

%RESTORE FRONTMATTER FOR FINAL VERSION
%\iffalse

\frontmatter

\title{
\textsf{\Huge \hbox{Introduction to Cluster Algebras}}\\[.1in]
\textsf{\Huge Chapter 6} \\[.2in]
{\rm\textsf{\LARGE (preliminary version)}}
}

\author{\Large \textsc{Sergey Fomin}}
%\address{Department of Mathematics, 
%University of Michigan,
%Ann Arbor, MI 48109} 
%\email{fomin@umich.edu}

\author{\Large \textsc{Lauren Williams}}
%\address{Department of Mathematics, University of California, Berkeley, CA 94720} 
%\email{williams@math.berkeley.edu}

\author{\Large \textsc{Andrei Zelevinsky}}
%\address{Department of Mathematics, Northeastern University, Boston, MA 02115}
%\email{andrei@neu.edu}

\maketitle

\noindent
\textbf{\Huge Preface}

\vspace{1in}

%\Large

\noindent
This is a preliminary draft of Chapter 6 of our forthcoming textbook
\textsl{Introduction to cluster algebras}, 
joint with Andrei Zelevinsky (1953--2013). 

Other chapters have been posted as 
\begin{itemize}[leftmargin=.2in]
\item  \texttt{arXiv:1608:05735} \hbox{(Chapters~1--3)}, 
\item \texttt{arXiv:1707.07190} (Chapters~4--5), 
%\item \texttt{arXiv:2008.09189} (Chapter~6), 
and
\item \texttt{arXiv:2106.02160} (Chapter~7). 
\end{itemize}
We expect to post additional chapters in the not so distant future. 

\medskip

We thank Bernard Leclerc and Karen Smith for their invaluable advice. 
The algebraic version of the argument 
in the proof of Proposition~\ref{prop:cluster-criterion} 
was suggested by Karen. 

We are grateful to Colin Defant, Chris Fraser, Anne Larsen, 
Amal Mattoo, Hanna \linebreak[3]
Mularczyk, and Raluca Vlad  
for a number of comments on the earlier versions of this chapter, 
and for assistance with creating figures.

Our work was partially supported by the NSF grants DMS-1664722 and
DMS-1854512.

\medskip

Comments and suggestions are welcome. 

\bigskip

\rightline{Sergey Fomin}
\rightline{Lauren Williams}

\vfill


\bigskip

\noindent
\copyright \ 2020--2021 by 
Sergey Fomin, Lauren Williams, and Andrei Zelevinsky

\tableofcontents

\mainmatter

\setcounter{chapter}{5}

%RESTORE FRONTMATTER FOR FINAL VERSION
%\fi

%\include{introduction}
%\include{tp-examples}
%\include{combinatorics-of-mutations}
%\include{geometric-type}
%\include{newfromold}
%\include{finite-type}
%\setcounter{page}{0}

% !TEX root = /Users/sf/Dropbox/Book/Sources/Ch1-8.tex
\chapter{Cluster structures in commutative rings}
\label{ch:rings}

\vspace{-1.5cm}

Cluster algebras are commutative rings endowed with a particular kind of combinatorial structure 
(a \emph{cluster structure}, as we call it). 
%A~cluster structure is what makes a ring a cluster algebra. 
In this chapter, we study the problem of identifying a cluster structure 
in a given commutative ring, or equivalently the problem of verifying that 
certain additional data make a given ring a cluster algebra.  

%Somewhat magically, many commutative rings arising ``in nature'' carry  a cluster structure. 
%leading to diverse applications of the general theory of cluster algebras. 

Sections~\ref{sec:coordinate-rings-small-rank}--\ref{sec:models-classical} 
provide several examples of cluster structures in coordinate rings of affine algebraic varieties. 
General techniques used to verify that a given commutative ring is a cluster algebra
are introduced in Section~\ref{sec:starfish}. \linebreak[3]
These techniques are applied in Sections~\ref{sec:rings-baseaffine}--\ref{sec:plucker-rings}
to treat several important classes of cluster algebras: 
the basic affine spaces for~$\SL_k$ (Section~\ref{sec:rings-baseaffine}),
the coordinate rings of $\operatorname{Mat}_{k \times k}$ 
and~$\SL_k$ (Section~\ref{sec:rings-matrices}),
and the homogeneous coordinate rings of Grassmannians,
also called Pl\"ucker rings (Section~\ref{sec:plucker-rings}). 
An~in-depth study of the latter topic will be given later in Chapter~\ref{ch:Grassmannians}, 
%where we will use the techniques (in particular the Starfish Lemma) introduced here.
following the development of the required combinatorial tools in Chapter~\ref{ch:plabic}. 
The problem of defining cluster algebras by generators and relations~is discussed in Section~\ref{sec:generators+relations}. 
%
%A brief and inevitably incomplete survey of other results 
%on cluster structures in commutative rings is given in Section~\ref{sec:rings-further}.

Section~\ref{sec:coordinate-rings-small-rank} is based on 
\cite[Section~12.1]{ca2} and \cite[Section~2]{gls-preprojective}. 
Sections~\ref{sec:cluster-algebras-and-coordinate-rings} and~\ref{sec:models-classical} 
follow \cite[Sections~11.1 and~12]{ca2}.  
Section~\ref{sec:starfish} follows~\cite[Section~3]{fomin-pylyavskyy}, \linebreak[3]
%Proposition~\ref{prop:cluster-criterion} (and 
%its proof given below) 
which in turn extends the ideas used in the proofs of 
\cite[Theorem~2.10]{ca3}
and subsequently \cite[Proposition~7]{scott}. 
The constructions 
%of cluster structures in coordinate rings of basic affine spaces (resp., spaces of square matrices) 
presented in Sections \ref{sec:rings-baseaffine} and~\ref{sec:rings-matrices} 
predate the general definition of cluster algebras; 
they essentially go back to~\cite{bfz-tp} and~\cite{fz-double}, respectively. 
Our development~of cluster 
structures in Grassmannians (Section~\ref{sec:plucker-rings}) 
is different from the original sources 
\cite{scott} and \cite[Section~3.3]{gsv-2003}. 
The material in Section~\ref{sec:generators+relations} is mostly new.

\newpage

\section{Introductory examples}
\label{sec:coordinate-rings-small-rank}

As a warm-up, we discuss several simple 
examples of cluster structures in commutative rings.

\begin{example}
Let $V = \CC^{2k}$, with $k\ge 3$, be an even-dimensional vector space with 
coordinates $(x_1,\dots, x_{2k})$.  Consider the nondegenerate
quadratic form $Q$ on $V$ given by 
\begin{equation}
\label{eq:quadric-2k}
Q(x_1,\dots, x_{2k}) = \sum_{i=1}^k (-1)^{i-1} x_i x_{2k+1-i}.
\end{equation}
Let $$\mathcal{C} = \{v \in V \ \vert \ Q(v)=0\}$$
be the isotropic cone and $\mathbb{P}(\mathcal{C})$ the 
corresponding smooth quadric in $\mathbb{P}(V)$.

The homogeneous coordinate ring of the quadric
(or equivalently the coordinate ring of $\mathcal{C}$) is the quotient 
\begin{equation}
\label{eq:quadric-ring}
\mathcal{A} = \CC[x_1,\dots, x_{2k}]/\langle Q(x_1,\dots, x_{2k})\rangle.
\end{equation}
To see that $\AA$ is a cluster algebra, we define, for $1\le s\le k-3$, the functions
$$p_s = \sum_{i=1}^{s+1} (-1)^{s+1-i} x_i x_{2k+1-i}.$$
% Note that we need to define $p_0$ for $s=0$ because it is used in the 
%exchange relation if $i=k-1$ and $k=3$.
Then $\mathcal{A}$ has cluster variables
$\{x_2, x_3,\dots, x_{k-1}\} \cup \{x_{k+2}, x_{k+3}, \dots, x_{2k-1}\}$
and frozen variables
$\{x_1, x_k, x_{k+1}, x_{2k}\} \cup \{p_s \ \vert \ 1 \leq s \leq k-3\}$.
It has $2^{k-2}$ clusters defined by choosing,
for each $i\in\{2,\dots, k-1\}$, precisely one of $x_i$ and $x_{2k+1-i}$.
The exchange relations are (here $2\le i\le k-1$): 
\begin{equation*}
x_i x_{2k+1-i} = 
\begin{cases}
p_{i-1}+p_{i-2} & \text{if $3 \leq i \leq k-2$;}\\
p_1 + x_1 x_{2k} & \text{if $i=2$ and $k\neq 3$;}\\
x_k x_{k+1} + x_1 x_{2k} & \text{if $i = 2$ and $k=3$}\\
x_k x_{k+1} + p_{k-3} & \text{if $i = k-1$ and $k\neq 3$}.
\end{cases}
\end{equation*}
This cluster algebra is of finite 
%rank~$k$ and 
 type~$A_1^{k-2} = A_1 \times A_1 \times \dots \times A_1$.
%Good reference is:\cite{gls-preprojective}.}
\cref{fig:quadric} shows a seed of~$\AA$ in the case $k=5$.

%\begin{figure}
%	\includegraphics[height=1.2in]{QuadricQuiver}
%	\caption{An initial seed for the cluster algebra $\mathcal{A}$ associated to the quadric in $\C^{10}$ 
%	(i.e. the case that $k=5$).
%	\Comment{TO DO: Use the code for 
%\cref{fig:gridquiver}
%	 as a model for how to 
%	make this figure.}}
%	\label{fig:quadric}
%\end{figure}

 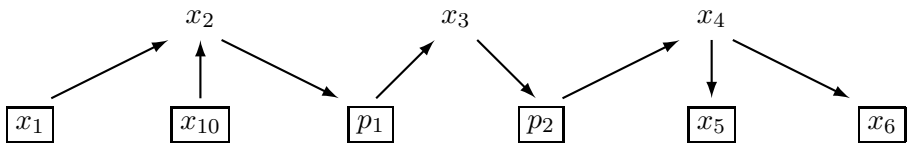
\begin{figure}[ht]
                \begin{center}
                        \setlength{\unitlength}{2pt}
                        \begin{picture}(160,20)(0,0)

                        %Boxes grouped by row
                        \put(32,20){\makebox(0,0){$x_{2}$}}
                        \put(80,20){\makebox(0,0){$x_{3}$}}
                        \put(128,20){\makebox(0,0){$x_{4}$}}

                        \put(0,0){\makebox(0,0){$\boxed{x_{1}}$}}
                        \put(32,0){\makebox(0,0){$\boxed{x_{10}}$}}
                        \put(64,0){\makebox(0,0){$\boxed{p_{1}}$}}
                        \put(96,0){\makebox(0,0){$\boxed{p_{2}}$}}
                        \put(128,0){\makebox(0,0){$\boxed{x_{5}}$}}
                        \put(160,0){\makebox(0,0){$\boxed{x_{6}}$}}

                        %Diagonal vectors
\thicklines
\dark{
                        \put(4,5){\vector(2,1){22}}
                        \put(32,5){\vector(0,1){11}}
                        \put(36,16){\vector(2,-1){22}}

                        \put(65,5){\vector(1,1){11}}
                        \put(84,16){\vector(1,-1){11}}

                        \put(100,5){\vector(2,1){22}}
                        \put(128,16){\vector(0,-1){11}}
                        \put(132,16){\vector(2,-1){22}}
		}
                        \end{picture}
                \end{center}
\caption{A seed for the cluster structure on the ring~\eqref{eq:quadric-ring},
for $k=5$.
%in $\mathbb{C}^{10}$ (i.e., $k=5$).
Here $p_1=-x_1x_{10}+x_2x_9\,$, $p_2=x_1x_{10}-x_2x_9+x_3x_8=x_4x_7-x_5x_6$. 
}
                \label{fig:quadric}
        \end{figure}

The quadric $\mathbb{P}(\mathcal{C})$ is a homogeneous space $G/P$ 
(a ``partial flag variety'') 
for the special orthogonal group attached to~$Q$.
The fact that~the coordinate ring $\CC[\mathcal{C}]$ is a cluster
algebra is a special case of a more general \hbox{phenomenon}. 
The (multi-)homogeneous coordinate ring of any type~$A$ partial flag \hbox{variety} 
carries a natural cluster algebra structure~\cite{gls-partialflag}. 
For generalizations to other semisimple Lie groups~$G$ and parabolic subgroups~$P\!\subset\! G$, 
see the survey~\cite{gls-survey-Lie}. 
\end{example}

\begin{example}
\label{ex:2x(n+1)}
Let $\AA=\CC[a_1,\dots,a_{n+1},b_1,\dots,b_{n+1}]$ be the coordinate ring 
of the affine space of $2\times (n+1)$ matrices
\begin{equation}
\label{eq:2x(n+1)}
\begin{bmatrix}
a_1&a_2&\cdots&a_{n+1}\\
b_1&b_2&\cdots&b_{n+1}
\end{bmatrix}.
\end{equation}
We will show that $\AA$ carries several pairwise non-isomorphic cluster algebra structures.

First, we can identify $\AA$ with a cluster algebra of type~$A_n$
as follows.  The cluster variables and frozen variables
are the $2(n+1)$ matrix entries
$a_1,\dots,a_{n+1},b_1,\dots,b_{n+1}$ together with the
$\binom{n+1}{2}$ minors
(Pl\"ucker coordinates)
$P_{ij}=a_ib_j-a_jb_i\,$.
The exchange relations are:
\begin{align*}
a_i \, b_j &= P_{ij}+ a_j\,b_i \quad (1\le i<j\le n+1),\\
a_j\,P_{ik} &= a_i\,P_{jk}+a_k\,P_{ij}\quad (1\le i<j<k\le n+1),\\
b_j\,P_{ik} &= b_i\,P_{jk}+b_k\,P_{ij}\quad (1\le i<j<k\le n+1),\\
P_{ik}\,P_{j\ell} &= P_{ij}\,P_{k\ell} +P_{i\ell}\,P_{jk}\quad (1 \leq i < j < k < \ell \leq n+1).
\end{align*}
By adding a column 
$\left[\begin{smallmatrix}
1\\  0
\end{smallmatrix}\right]$
%$(1 0)$ 
at the beginning
and a column 
$\left[\begin{smallmatrix}
0\\  1
\end{smallmatrix}\right]$
%$(0 1)$ 
at the end of the $2 \times (n+1)$ matrix,
	we obtain a full rank $2 \times (n+3)$ matrix,
	which we can view as an element of the Grassmannian 
	${\rm Gr}_{2,n+3}$.  
	Under this identification, matrix entries 
	of the $2 \times (n+1)$ matrix 
	are equal to Pl\"ucker coordinates of the corresponding 
	element of 
	${\rm Gr}_{2,n+3}$:  
	 $a_i = P_{i+1,n+3}$ and
	$b_i = P_{1, i+1}$.  Note also that $P_{1,n+3}=1$.  We can thus 
	identify $\AA$ with the quotient $R_{2,n+3}/\langle P_{1, n+3} - 1\rangle$
of the Pl\"ucker ring $R_{2,n+3}$.  Our cluster variables and frozen variables
	for $\AA$ are inherited from the cluster structure on $R_{2,n+3}$.  
The cluster algebra $\AA$ has rank $n$, with  $n+2$ frozen variables.
In the case $n=1$ we
recover  Example~\ref{ex:SL_2}.

On the other hand, 
subdividing a $2 \times (n+1)$ matrix~\eqref{eq:2x(n+1)}
into a $2 \times i$ matrix and a $2 \times (n+1-i)$ matrix,
%(where we use all matrix entries and $2 \times 2$ minors
%contained in one or the other matrix),
we can make~$\AA$ into a cluster algebra of type $A_{i-1} \times A_{n-i}$.
The cluster and frozen variables would 
include all matrix entries as well as every $2 \times 2$ minor
contained in one of the two distinguished submatrices.
The total number of frozen variables in this cluster algebra is 
$(i+1) + (n+2-i) = n+3$.
For each $i$, this gives a cluster algebra structure 
of rank $(i-1)+(n-i) = n-1$, with $n+3$ frozen variables.

More generally, we can partition a $2 \times (n+1)$ matrix~\eqref{eq:2x(n+1)}
into $k$ matrices of sizes
$2 \times i_1$, \dots , $2 \times i_k$, 
where $i_1 + \dots + i_k = n+1$. \linebreak[3]
The cluster variables and frozen variables would  
include all matrix entries plus every $2 \times 2$ minor
contained in one of the $k$ distinguished submatrices.
This gives a cluster structure of type 
$A_{i_1-1} \!\times\! \dots\! \times \!A_{i_k-1}$ on the ring~$\AA$.
This cluster algebra has rank
$(i_1\!-\!1)\! +\! (i_2\!-\!1)\! +\! \dots \!+\! (i_k\!-\!1)\! =\! n\!-\!k\!+\!1$, 
and has $(i_1+1)+(i_2+1)+\dots + (i_k+1) = n\!+\!k\!+\!1$ frozen variables.
\end{example}

\iffalse 
\begin{example}[cf.\ Example~\ref{ex:SL_2}]
The coordinate ring of the affine~space of $2 \times 2$ complex matrices 
$\matbr{a}{b}{c}{d}$ is the polynomial ring $\CC[a,b,c,d]$.  
This ring can be viewed as a cluster algebra of rank~1 in which
\begin{itemize}[leftmargin=.2in]
\item
the ambient field is $\FFcal=\CC(a,b,c,d)$;
\item 
the frozen variables are $b$, $c$, and $\Delta=ad-bc$;
\item
the cluster variables are $a$ and~$d$;
\item
the single exchange relation is
$ad = \Delta + bc$. 
%(cf.\ \eqref{eq:2by2-3term-relation}).
\end{itemize}

There is also a cluster structure of type $A_0$ here 
where all variables are frozen.  

A variation of the previous example obtained by setting
$\Delta = 1$   yields a cluster structure of type $A_1$ on 
the coordinate ring of 
$\operatorname{SL}_2(\CC)$.
\end{example}
\fi

%\newpage

\section{Cluster algebras and coordinate rings} 
\label{sec:cluster-algebras-and-coordinate-rings}

Suppose a collection of regular functions
on an algebraic variety~$X$ satisfies relations 
which can be interpreted as exchange relations for a seed pattern. 
Then---subject to conditions articulated below---the coordinate ring of~$X$ can be
naturally identified with the corresponding cluster algebra:  
 
\begin{proposition} 
\label{pr:geom-realization-criterion} 
Let $\AA$ be a cluster algebra 
(of geometric type, over~$\CC$) of rank~$n$, 
with frozen variables $x_{n+1},\dots,x_m$. 
Let $\mathcal{X}$ denote the set of cluster variables in~$\AA$. 
Let $X$ be a rational affine irreducible algebraic variety 
of dimension~$m$. 
%satisfying conditions~\eqref{eq:X-rational-etc}. 
Suppose we are given a family of nonzero regular functions 
\[
\{\varphi_z : z \in \mathcal{X}\} \cup 
\{\varphi_{n+1},\dots,\varphi_m\}\subset \CC[X]
\]
satisfying the following conditions: 
\begin{eqnarray} 
%\begin{equation} 
%\label{eq:X-dimension} 
%&&\dim (X) = n + |J|;\\ 
%\end{equation} 
%\begin{equation} 
\label{eq:phi-generate-X} 
&&\text{the functions $\varphi_z$ ($z\in \mathcal{X}$) and 
$\varphi_i$ ($n+1 \le i\le m$)  
generate~$\CC[X]$;}\\ 
%\end{equation} 
%\begin{equation} 
\label{eq:same-relations} 
&&\text{replacing each cluster variable $z$ by~$\varphi_z$,
and each frozen variable~$x_i$}\\ 
\nonumber 
&&\text{by~$\varphi_i$ makes every exchange relation \eqref{eq:exch-rel-geom} into 
an identity in $\CC[X]$.} \hspace{-.2in} 
\end{eqnarray} 
Then there is a unique $\CC$-algebra isomorphism $\varphi:\AA\to\CC[X]$
such that \linebreak[3]
$\varphi(z)=\varphi_z$ for all $z\in\mathcal{X}$
and $\varphi(x_i)=\varphi_i$ for $i\in\{n+1,\dots,m\}$. 
\end{proposition} 

\iffalse
\begin{equation} 
\label{eq:varphi} 
z \mapsto \varphi_z \ \  (z \in \mathcal{X}), 
\quad x_i \mapsto \varphi_i \ \  (n<i\le m) 
\end{equation} 
extends uniquely to an algebra isomorphism between the cluster algebra 
$\AA$ and the coordinate ring~$\CC[X]$. 
\fi

\begin{remark}
We briefly comment on the general assumptions on the variety~$X$
made above. %in Proposition~\ref{pr:geom-realization-criterion}. 
%\begin{equation} 
%\label{eq:X-rational-etc} 
%\text{$X$ is a rational quasi-affine irreducible algebraic variety 
%over $\CC$.} \hspace{.2in} 
%\end{equation} 
Irreducibility implies that the ring of regular functions $\CC[X]$ is a 
domain, so its fraction field is well defined 
%Quasi-affine means Zariski open in some affine variety; 
%this ensures that 
(and coincides with the field $\CC(X)$ of
rational functions on~$X$). 
Rationality of~$X$ means that $\CC(X)$ 
is isomorphic to the field of rational functions 
over~$\CC$ in $\dim (X)$ independent~variables. 
In~a typical application, $X$~contains an open
subset isomorphic to an affine space, 
so this condition is satisfied. 
\end{remark}
 
\begin{proof}
The key assertion to be proved is that each cluster in $\mathcal{A}$
gives rise to a transcendence basis of the field of rational functions
$\CC(X)$.
Pick a seed in~$\AA$; let $\xx$ (resp.,~$\tilde\xx$) be the
corresponding cluster (resp., extended cluster). 
%Any extended cluster $\tilde \xx$ of~$\AA$ is a transcendence basis
%of the ambient field~$\FFcal$ over~$\CC$.
Every  cluster variable $z\in\mathcal{X}$ 
is expressed as a rational function in 
$\tilde \xx$ by iterating the exchange relations away from the chosen seed. 
%Since $X$ is irreducible, 
By~\eqref{eq:same-relations}, we can apply the same 
procedure to express all functions $\varphi_z$ and $\varphi_i$ 
inside the field $\CC(X)$ as rational functions in the set 
\[
\Phi = \{\varphi_x: x \in \xx\} \cup \{\varphi_{n+1},\dots,\varphi_m\}.
\] 
Since $X$ is rational and $|\Phi| =m= \dim (X)$, we conclude from (\ref{eq:phi-generate-X}) 
that $\Phi$ is a transcendence basis of the field 
of rational functions $\CC(X)$, and that the correspondence  
\[
z \mapsto \varphi_z \ \  (z \in \mathcal{X}), 
\quad x_i \mapsto \varphi_i \ \  (n<i\le m) 
\]
extends uniquely to an isomorphism of fields $\FFcal \to \CC(X)$, and hence yields an 
isomorphism of algebras $\AA \to \CC[X]$. 
\end{proof}

\section{Examples of cluster structures of classical types 
%$A_n$, $B_n$, $C_n$, and $D_n$
}
\label{sec:models-classical}

Informally speaking,
Proposition~\ref{pr:geom-realization-criterion} tells us that in order to
identify a coordinate ring of a rational algebraic variety as a cluster algebra, 
it suffices to find elements of that ring that satisfy the requisite
exchange relations. 
In reality, this approach is only practical for cluster algebras of
finite type. 
In~this section, we present four examples of coordinate rings 
endowed with cluster structures of types $A_n$, $B_n$, $C_n$, and~$D_n$,
respectively. 
All four rings are closely related to each other;
the first two of them are actually identical (as commutative rings) even
though the cluster structures are different.

Our first example, the homogeneous coordinate ring of a 
Grassmannian of $2$-planes, has already been thoroughly examined 
in Sections~\ref{sec:Ptolemy} and~\ref{sec:type-A}. 

\begin{example} 
[\emph{Type~$A_n$}] 
\label{ex:realization-An} 
%We will now give a geometric realization of $\AA_\circ\,$. 
Let $X=\widehat{\Gr}_{2, n+3}$ 
be the affine cone over the Grassmannian ${\rm Gr}_{2,n+3}$ 
of 2-dimensional subspaces in $\CC^{n+3}$ 
taken in its Pl\"ucker embedding.  
Equivalently, $X$ can be viewed as the variety of nonzero decomposable bivectors: 
\[
X\cong\{u\wedge v\neq 0 \mid u,v\in\CC^{n+3}\}.
\] 
This is a $(2n+3)$-dimensional affine algebraic variety. 
Its coordinate ring is the Pl\"ucker ring $R_{2,n+3}=\CC[X]$. 
This ring is generated by the standard Pl\"ucker coordinates $P_{ab}\in\CC[X]$, for 
$1 \le a < b \leq n+3$. 

Alternatively, we can view the Pl\"ucker ring $R_{2,n+3}$ as
the ring of $\SL_2$-invariant polynomial functions 
on the space of $(n+3)$-tuples of vectors in~$\CC^2$. 
Representing these vectors as columns of a $2 \times (n+3)$ 
matrix $z= (z_{ab})$, one identifies the Pl\"ucker coordinates with 
the $2 \times 2$ minors of~$z$: 
\[
P_{ab} = z_{1a} z_{2b} - z_{1b} z_{2a} \quad (1 \leq a < b \leq
n+3).
\]

In Section~\ref{sec:type-A}, we constructed a seed pattern 
of type~$A_n$ in the field of rational functions~$\CC(X)$. 
The seeds in this pattern are labeled by triangulations of a
convex $(n+3)$-gon $\mathbf{P}_{n+3}$ by pairwise noncrossing diagonals.
Each cluster consists of the Pl\"ucker
coordinates $P_{ij}$ corresponding to the diagonals in a
given triangulation. 
The frozen variables are the Pl\"ucker coordinates associated with the sides of~$\mathbf{P}_{n+3}$. 
The exchange relations of the seed pattern are exactly the Grassmann-Pl\"ucker relations~\eqref{eq:grassmann-plucker-3term}. 
Thus, we can view this example as an instance of \cref{pr:geom-realization-criterion}. 

\pagebreak[3]

As a cluster algebra of type~$A_n$, the Pl\"ucker ring $R_{2,n+3}$ 
is generated by the cluster and coefficient variables, 
which are precisely the Pl\"ucker coordinates~$P_{ij}$.
It is moreover well known (and not hard to see) 
that the ideal of relations among the Pl\"ucker coordinates
is generated by the exchange relations, or more precisely by the polynomials 
\[
P_{ik}P_{jl}-P_{ij}P_{kl}-P_{il}P_{jk} \quad (1\le i<j<k<l\le n+3)
\]
(cf.~\eqref{eq:grassmann-plucker-3term}). 
We will see in Section~\ref{sec:generators+relations} 
that this phenomenon does not hold in general:
even when a cluster algebra is of finite type, some of the relations among its generators
may not lie in the ideal generated by the exchange relations. 
\end{example}

While the type~$A$ cluster structure on a Pl\"ucker ring
$R_{2,m}$ is perhaps the most natural one, we can also endow this ring 
with a type~$B$ cluster structure, as we now explain.

\begin{example} 
[\emph{Type~$B_n$}] 
\label{ex:realization-Bn} 
The two-element group 
$\ZZ/2\ZZ$ acts on the set of tagged arcs and boundary segments in 
the punctured polygon $\mathbf{P}_{n+1}^\bullet$ 
(see Definition~\ref{def:arcs-Dn}) by switching the tagging on radii, and leaving everything else intact.
Let us associate an element~$P_\gamma$ of the Pl\"ucker ring $R_{2,n+2}$ to every
$\ZZ/2\ZZ$-orbit~$\gamma$ %in the set of tagged arcs and boundary segments in 
%the punctured polygon $\mathbf{P}_{n+1}^\bullet$, 
as follows
(cf.\ Definition~\ref{def:P-gamma}): 
\[
P_\gamma\!=\!
\begin{cases}
%\langle v_i,v_j \rangle 
P_{ab}
&\text{if $\gamma$ doesn't cross the cut, and has endpoints $a$ and $b>a$;}\\
%\langle v_j,Av_i \rangle 
P_{a \bar b}
&\text{if $\gamma$ crosses the cut, and has endpoints $a$ and $b>a$;}\\
%\langle v_i,a \rangle 
P_{a,n+2} 
&\text{if $\gamma$ is an orbit of radii with endpoints $p$ and~$a$,}
\end{cases}
\]
%For $1\le i<j\le n+1$, let us denote
where we use the notation 
\begin{equation}
\label{eq:P-a,bar-b}
P_{a \bar b}=P_{a,n+2} P_{b,n+2} - P_{ab} \,. 
\end{equation}
The cluster variables and frozen variables are the elements $P_{\gamma}$,
where $\gamma$ ranges over orbits of tagged arcs and boundary segments, respectively.

We use Proposition~\ref{pr:geom-realization-criterion} to show that 
this yields a cluster structure of type~$B_n$ in~$R_{2,n+2}$. 
The only nontrivial task is to check condition 
\eqref{eq:same-relations}, which amounts to verifying the 
following six identities: % (with $r=1$): 
%obtained from the exchange relations 
%(\ref{eq:exchange-relation-C-1})--(\ref{eq:exchange-relation-C-3}) 
%(we have to take into account possible positions of vertices 
%in Figure~\ref{fig:quadrilateral-D} among the vertices 
%$1, \dots, n+1, \overline 1, \dots, \overline {n+1}$): 
\begin{align} 
\label{eq:exchange-relations-in-deltas-Cn-1} 
P_{ac} \, P_{bd} = 
P_{ab}\, P_{cd} + 
P_{ad}\, P_{bc}  \qquad & (1 \leq a < b < c < d \leq n+1),\\ 
\label{eq:exchange-relations-in-deltas-Cn-2} 
P_{a \overline c} \, P_{bd} = 
P_{a \overline b}\, P_{cd} + 
P_{a \overline d}\, P_{bc} \qquad & (1 \leq a < b < c < d \leq n+1),\\ 
\label{eq:exchange-relations-in-deltas-Cn-3} 
P_{a \overline c} \, P_{b \overline d} = 
P_{ab}\, P_{cd} + 
P_{a \overline d}\, P_{b \overline c} \qquad & (1 \leq a < b 
< c < d \leq n+1),\\ 
%\label{eq:exchange-relations-in-deltas-Bn-4} 
%P_{ac} \, P_{a \overline b} = 
%P_{ab}\, P_{a \overline c} + 
%P_{a \overline a}^{2/r} \, P_{bc} \qquad & (1 \leq a < b < c \leq n+1),\\ 
%\label{eq:exchange-relations-in-deltas-Bn-5} 
%P_{a \overline b} \, P_{b \overline c} = 
%P_{ab}\, P_{bc} + 
%P_{b \overline b}^{2/r} \, P_{a \overline c} \qquad & 
%(1 \leq a < b < c \leq n+1),\\ 
%\label{eq:exchange-relations-in-deltas-Bn-6} 
%P_{a\overline a}\, P_{b \overline b} = 
%P_{ab}^r + P_{a \overline b}^r \qquad\qquad &(1 \leq a < b \leq n+1) \, . 
\label{eq:exchange-relations-in-deltas-Cn-4} 
P_{ac} \, P_{a \overline b} = 
P_{ab}\, P_{a \overline c} + 
P_{a, n+2}^{2} \, P_{bc} \qquad & (1 \leq a < b < c \leq n+1),\\ 
\label{eq:exchange-relations-in-deltas-Cn-5} 
P_{a \overline b} \, P_{b \overline c} = 
P_{ab}\, P_{bc} + 
P_{b, n+2}^{2} \, P_{a \overline c} \qquad & 
(1 \leq a < b < c \leq n+1),\\ 
\label{eq:exchange-relations-in-deltas-Cn-6} 
%P_{a\overline a}\, P_{b \overline b} = 
%P_{ab}^2 + P_{a \overline b}^2 
P_{a,n+2} P_{b,n+2} = P_{ab} + P_{a\bar b}
\qquad &(1 \leq a < b \leq n+1) \, . 
\end{align} 
While these identities %, (\ref{eq:exchange-relations-in-deltas-Bn-1})
                       %is 
can be directly deduced from the Grassmann-Pl\"ucker relations, 
%, and the rest are reduced 
%to this relation by simple algebraic manipulations. 
we prefer another route, 
presented below in a slightly informal~way. 

Consider the algebraic model of a seed pattern of type~$D_{n+1}$ 
described in Section~\ref{sec:type-D}. 
(Note that we are replacing $n$ by $n+1$.) 
Recall that it involves working with $n+1$ two-dimensional vectors $v_1,\dots,v_{n+1}$, 
two ``special'' vectors $a$ and~$\overline{a}$, 
and two scalars $\lambda$ and~$\overline{\lambda}$.
To get a seed pattern of type~$B_n$, we specialize the type~$D_{n+1}$
seed pattern as follows. 
The vectors $v_1,\dots,v_{n+1}$ are kept without change. 
We take vectors $a$ and $b$ satisfying 
\begin{equation}
\label{eq:ba=1}
\langle b,a\rangle=1  
\end{equation}
(we shall later identify $a$ with~$v_{n+2}$), 
set
\begin{align*}
\overline{a}&=a+\varepsilon b,\\
\lambda&=1,\\
\overline{\lambda}&=1+\varepsilon,
\end{align*}
and take the limit $\varepsilon\to 0$. 
We then have  
\[
a^{\notch}
=\frac{\overline{\lambda}-\lambda}{\langle
  \overline{a}, a\rangle} \,\overline{a}
=\frac{\varepsilon}{\langle
  a+\varepsilon b, a\rangle}\,(a+\varepsilon b)
\to a,
\]
so in the limit we get $a^{\notch}=a$ and
$\overline{\lambda}=\lambda$. 
This yields the folding of our type $D_{n+1}$ seed pattern into
a type $B_n$ pattern. 
It remains to check that the (folded) cluster variables of the type
$D_{n+1}$ pattern specialize to the cluster variables~$P_\gamma$
defined above. 
The only nontrivial case is the second one, wherein 
$P_\gamma=P_{a\bar b}$. 
The operator $A$ defined in~\eqref{eq:Adefn} specializes via
\[
A v = \frac{\overline{\lambda} \langle v,a \rangle \,\overline{a} - \lambda \langle
  v,\overline{a} \rangle \,a}{\langle \overline{a}, a\rangle}
\to 
\langle v,a\rangle a - \langle v,b\rangle a+\langle v,a\rangle b. 
\]
As a result, we get, %in the limit $\varepsilon\to 0$:
using the Grassmann-Pl\"ucker relation and~\eqref{eq:ba=1}: 
\[
\langle w,Av\rangle 
\to \langle v,a\rangle\langle w,a\rangle - \langle v,b\rangle\langle
w,a\rangle + \langle v,a\rangle \langle w,b\rangle
=\langle v,a\rangle\langle w,a\rangle-\langle v,w\rangle, 
\]
So in particular $\langle v_j,Av_i \rangle\to \langle
v_i,a\rangle\langle v_j,a\rangle-\langle v_i,v_j\rangle$, 
matching~\eqref{eq:P-a,bar-b}. 
\end{example} 

\begin{remark}
%\label{rem:}
Examples~\ref{ex:2x(n+1)} and~\ref{ex:realization-An}-\ref{ex:realization-Bn} demonstrate 
that a given ring can carry different non-isomorphic cluster structures. 
%It can moreover happen that the same ring can be viewed as a cluster algebra of finite type 
%or of infinite type, or even an infinite mutation type. 
%One example of this kind 
A~particularly striking example was given in \cite[Figure~20]{fomin-pylyavskyy}:
the ``mixed Pl\"ucker~ring" $\CC[V^3\times (V^*)^4]^{\SL(V)}$, 
with $V\cong\CC^3$, can carry a cluster structure of finite type $D_6$ or~$E_6$, 
or a cluster structure of an infinite mutation type (hence of infinite type). 
\end{remark}
 
\begin{example} 
[\emph{Type~$C_n$}] 
\label{ex:realization-Cn} 
Let $\operatorname{SO}_2$ be the group of complex matrices 
\[ 
\bmat{u}{-v}{v}{u} 
\] 
with $u^2+v^2 = 1$. 
Consider the ring $R = \CC[V^{n+1}]^{\operatorname{SO}_2}$ 
of $\operatorname{SO}_2$-invariant polynomial functions 
on the space of $(n+1)$-tuples of vectors
\begin{equation}
\label{eq:v1-vn+1}
(v_1,\dots,v_{n+1})\in V^{n+1},\quad V=\CC^2,
\end{equation}
or equivalently $\operatorname{SO}_2$-invariant polynomials
in the entries of a $2\times (n+1)$ matrix
\[
z=\begin{bmatrix}
z_{11} & \cdots & z_{1,n+1}\\ 
z_{21} & \cdots & z_{2,n+1}
\end{bmatrix}. 
\]
This ring is generated by the Pl\"ucker coordinates 
\[
%\label{eq:Delta-functions-Cn-1} 
%\{[a,b], [\overline a, \overline b]\} \mapsto 
P_{ab} 
= \langle v_a,v_b\rangle
= z_{1a} z_{2b} - z_{1b} z_{2a} 
%= \displaystyle\frac{y_{1a}y_{2b} - y_{1b}y_{2a}}{2i} 
\quad (1 \leq a < b \leq n+1)
\]
together with the polynomials
\[
%\label{eq:Delta-functions-Cn-2} 
%\{[a,\overline b], [\overline a, b]\} \mapsto 
P_{a \overline b} 
= \langle v_b,Mv_a\rangle
= z_{1a} z_{1b} + z_{2a} z_{2b} 
%= \displaystyle\frac{y_{1a}y_{2b} + y_{1b}y_{2a}}{2} 
\quad (1 \leq a \leq b \leq n+1) \,,  
\]
where $M=\left[\begin{smallmatrix}
0 & -1\\
1 & 0
\end{smallmatrix}\right]\in\operatorname{SO}_2$. 
\iffalse
Alternatively, $R$~can be identified with the ring 
of invariants 
$\CC[V^{n+1}]^T$ where $T\subset \SL_2$ is the torus of all 
diagonal matrices of the form 
\[ 
\bmat{t}{0}{0}{t^{-1}}. 
\] 
Indeed, we have $g(\operatorname{SO}_2) g^{-1} = T$, where 
\[ 
g = \bmat{1}{-i}{1}{i}, 
\] 
so the map $f \mapsto f^g$  defined by 
$f^g(z) = f(gz)$ is an isomorphism 
\begin{equation} 
\label{eq:SO2=T} 
\CC[V^{n+1}]^T \to \CC[V^{n+1}]^{\operatorname{SO}_2}. 
\end{equation} 
The ring $R=\CC[V^{n+1}]^T$ can 
\fi
%
The ring $R=\CC[V^{n+1}]^{\operatorname{SO}_2}$ can 
also be viewed as the coordinate ring $\CC[X]$ of the variety 
$X$ of complex $(n+1) \!\times\! (n+1)$ matrices of rank $\leq 1$; 
even more geometrically, $X$~is the affine cone over the 
product of two copies of the projective space $\mathbb{CP}^n$ taken in the 
Segre embedding. 
Specifically, the~map 
\iffalse
\[ 
y=\left[\!\!\begin{array}{ccc} 
y_{11} & \cdots & y_{1,n+1}\\ 
y_{21} & \cdots & y_{2,n+1}\\ 
\end{array}\!\!\right] \,\, 
\mapsto \,\, (y_{1a}y_{2b})_{a, b= 1, \dots, n+1} \in X 
\] 
induces an algebra isomorphism 
$\CC[X] \to \CC[V^{n+1}]^T$. 
Combining this with (\ref{eq:SO2=T}), we obtain an isomorphism 
$\CC[X] \to \CC[V^{n+1}]^{\operatorname{SO}_2}$ induced by the map 
%${\rm Mat}_{2,n+1} \to X$ given by 
\fi
\[ 
z=\begin{bmatrix}
z_{11} & \cdots & z_{1,n+1}\\ 
z_{21} & \cdots & z_{2,n+1}
\end{bmatrix}
\mapsto ((z_{1a} - i z_{2a})(z_{1b} + i z_{2b}))_{a,b\in\{1,\dots,n+1\}}
%=(P_{a\bar b} + i P_{ab})
\in X 
\] 
induces an algebra isomorphism 
$\CC[X] \to \CC[V^{n+1}]^{\operatorname{SO}_2}$. 
(Note that 
\[
(z_{1a} - i z_{2a})(z_{1b} + i z_{2b})=P_{a\bar b} + i P_{ab}\,.)
\]

To construct a cluster algebra structure of type~$C_n$ in this ring,
let us 
associate an element~$P_\gamma\in R$ to every orbit $\gamma$ of the action 
of 
$\ZZ/2\ZZ$ on the set of diagonals and sides of a 
regular $(2n+2)$-gon 
 $\mathbf{P}_{2n+2}$.  
Specifically,  we~set
%as follows
%(cf.\ Definition~\ref{def:P-gamma}): 
\[
P_\gamma\!=\!
\begin{cases}
%\langle v_i,v_j \rangle 
P_{ab}
&\text{if $\gamma = 
\{(a,b), (a+n+1, b+n+1)\}$ for $a<b \leq n+1$;}\\
%\langle v_j,Av_i \rangle 
P_{a \bar b}
&\text{if $\gamma = 
\{(a,b+n+1), (a+n+1,b)\}$ for $a\leq b \leq n+1$.}
\end{cases}
\]
%For $1\le i<j\le n+1$, let us denote
where we use the notation $(a,b)$ to denote 
a diagonal or side with endpoints $a$ and $b$.
The cluster variables and frozen variables are the elements $P_{\gamma}$,
where $\gamma$ ranges over orbits of diagonals and boundary segments, respectively.

%A cluster algebra structure of type~$C_n$ in this
%ring is constructed by associating 
%an element $P_\gamma\in R$ to every orbit~$\gamma$ 
%of the action of $\ZZ/2\ZZ$ on the set of 
%diagonals and sides of a regular $(2n+2)$-gon $\mathbf{P}_{2n+2}\,$.
%Specifically, we set $P_\gamma=P_{ab}$ (resp., $P_\gamma=P_{a\bar b}$)
%if $\gamma$ includes line segments with
%endpoints $\{a,b\}$ and $\{a+n+1,b+n+1\}$
%(resp., $\{a,b+n+1\}$ and $\{a+n+1,b\}$). 
 
The verification is similar to Example~\ref{ex:realization-Bn} above. 
The only substantive task is to check that the functions $P_{ab}$ and
$P_{a\bar b}$ satisfy the requisite exchange relations:
\begin{align} 
\label{eq:exchange-relations-in-deltas-Bn-1} 
P_{ac} \, P_{bd} = 
P_{ab}\, P_{cd} + 
P_{ad}\, P_{bc}  \qquad & (1 \leq a < b < c < d \leq n+1),\\ 
\label{eq:exchange-relations-in-deltas-Bn-2} 
P_{a \overline c} \, P_{bd} = 
P_{a \overline b}\, P_{cd} + 
P_{a \overline d}\, P_{bc} \qquad & (1 \leq a < b < c < d \leq n+1),\\ 
\label{eq:exchange-relations-in-deltas-Bn-3} 
P_{a \overline c} \, P_{b \overline d} = 
P_{ab}\, P_{cd} + 
P_{a \overline d}\, P_{b \overline c} \qquad & (1 \leq a < b 
< c < d \leq n+1),\\ 
\label{eq:exchange-relations-in-deltas-Bn-4} 
P_{ac} \, P_{a \overline b} = 
P_{ab}\, P_{a \overline c} + 
P_{a \overline a} \, P_{bc} \qquad & (1 \leq a < b < c \leq n+1),\\ 
\label{eq:exchange-relations-in-deltas-Bn-5} 
P_{a \overline b} \, P_{b \overline c} = 
P_{ab}\, P_{bc} + 
P_{b \overline b} \, P_{a \overline c} \qquad & 
(1 \leq a < b < c \leq n+1),\\ 
\label{eq:exchange-relations-in-deltas-Bn-6} 
P_{a\overline a}\, P_{b \overline b} = 
P_{ab}^2 + P_{a \overline b}^2 \qquad\qquad &(1 \leq a < b \leq n+1) \, . 
\end{align} 
This can either be done directly
 or via folding,
this time going from a cluster structure of type $A_{2n-1}$ in the
Pl\"ucker ring $R_{2,2n+2}$ to the type~$C_n$ cluster structure in~$\CC[X]$, as follows. 
Starting with the $(n+1)$-tuple~\eqref{eq:v1-vn+1}, 
we build the $(2n+2)$-tuple $(v_1,\dots,v_{2n+2})$ 
by setting $v_{n+1+a}=Mv_a$ for $a\in\{1,\dots,n+1\}$. 
In this specialization, 
using the fact that $M^2=-1$, we have
\begin{align*}
\langle v_{n+1+a}, v_{n+1+b} \rangle
&=\langle M v_{a}, M v_{b} \rangle
=\langle v_{a}, v_{b} \rangle,\\
\langle v_{b}, v_{n+1+a} \rangle
&=\langle v_{b}, M v_{a} \rangle
%=\langle Mv_{b}, M^2 v_{a} \rangle
=\langle Mv_{b}, -v_{a} \rangle
=\langle v_{a}, v_{n+1+b} \rangle.
\end{align*}
This shows that two Pl\"ucker coordinates corresponding
to centrally symmetric diagonals in 
 $\mathbf{P}_{2n+2}$ are equal when evaluated at 
the $(2n+2)$-tuple $(v_1,\dots,v_{2n+2})$.   
Thus, the elements $P_{\gamma}\in \CC[X]$ defined earlier come from Pl\"ucker
coordinates in $R_{2,2n+2}$ via folding.
The exchange relations in question are obtained by specializing the exchange relations in the Pl\"ucker ring. 
\iffalse
This is completely straightforward; for example, 
(\ref{eq:exchange-relations-in-deltas-Bn-6}) becomes 
\[ 
(z_{1a}^2 + z_{2a}^2)(z_{1b}^2 + z_{2b}^2) = 
(z_{1a} z_{2b} - z_{1b} z_{2a})^2 + 
(z_{1a} z_{1b} + z_{2a}z_{2b})^2 
%(y_{1a}y_{2a})(y_{1b}y_{2b}) 
%=\Bigl(\displaystyle\frac{y_{1a}y_{2b} - y_{1b}y_{2a}}{2i}\Bigr)^2 
%+\Bigl(\displaystyle\frac{y_{1a}y_{2b} + y_{1b}y_{2a}}{2}\Bigr)^2 
.
\] 
\fi
\end{example} 

\begin{example} 
[\emph{Type~$D_n$}] 
\label{ex:realization-Dn} 
Let $\widehat{\Gr}_{2, n+2}$ 
 be the affine cone over the Grassmannian ${\rm
  Gr}_{2,n+2}$, taken in its Pl\"ucker embedding. 
%as in Example~\ref{ex:realization-Bn}. 
Let $X$ be the ``Schubert'' divisor in $\widehat{\Gr}_{2,n+2}$ 
given by the equation 
$P_{n+1,n+2}= 0$; thus, we have 
$$\CC[X] = \CC[\widehat{\Gr}_{2,n+2}]/\langle P_{n+1,n+2} \rangle \ .$$ 
%(Geometrically, $X$ is the affine cone over 
%the Schubert divisor in the Grassmannian ${\rm Gr}_{2,n+2}$.) 
A cluster structure of type~$D_n$ in the coordinate ring $R=\CC[X]$ can be obtained by
associating an element $P_\gamma\in R$ to each tagged arc $\gamma$
in the punctured polygon~$\mathbf{P}_n^\bullet$, as follows
(cf.\ Definition~\ref{def:P-gamma}):
\[
P_\gamma\!=\!
\begin{cases}
P_{ab}
&\text{if $\gamma$ doesn't cross the cut} \\
& \qquad \text{and has endpoints $a$ and $b>a$;}\\
%\langle v_j,Av_i \rangle 
P_{a,n+1} P_{b,n+2} - P_{ab} 
&\text{if $\gamma$ crosses the cut} \\
& \qquad \text{and has endpoints $a$ and $b>a$;}\\
%\langle v_i,a \rangle 
P_{a,n+1}
&\text{if $\gamma$ is a plain radius with endpoint~$a$;}\\
%\langle v_i,a^{\notch} \rangle 
P_{a,n+2}
&\text{if $\gamma$ is a notched radius with endpoint~$a$.}
\end{cases}
\]
(We thus identify the two eigenvectors of~\eqref{eq:Adefn} 
with $v_{n+1}$ and~$v_{n+2}$, respectively.)
The verification is left to the reader; or see \cite[Example~12.15]{ca2}. 
\end{example} 
 
\begin{remark}
While the seed pattern of type~$D_n$ described in
Example~\ref{ex:realization-Dn} is much simpler than the one used in
Section~\ref{sec:type-D}, 
we did not use it there because the corresponding exchange matrices do
not have full rank. 
\end{remark}

%\pagebreak[3]
%\newpage

\section{Starfish lemma}
\label{sec:starfish}

In what follows, we denote by $\AA(\tilde\xx,\tilde B)$ the cluster
algebra defined by a seed $(\tilde\xx,\tilde B)$ 
in some ambient field of rational functions freely generated by~$\tilde\xx$. 

Any cluster algebra, being a subring of a field, is an integral
domain (and under our conventions, a \hbox{$\CC$-algebra}). 
Conversely, given such a domain, one may be interested in
identifying it as a cluster algebra. 

For the remainder of this section, we let $\Rcal$ be an integral domain and a
$\CC$-algebra, and we  
denote by $\FFcal$ the quotient field of~$\Rcal$. 
%As an ambient field~$\FFcal$, 
%we can use the quotient field~$\QF(\Rcal)$. 
The challenge is to find a seed $(\tilde\xx,\tilde B)$ in $\FFcal$ such
that $\AA(\tilde\xx,\tilde B)=\Rcal$. 
The difficulties here are two-fold.
To prove the inclusion $\AA(\tilde\xx,\tilde B)\supset \Rcal$,
we need to demonstrate that (a subset of) cluster variables in this
seed pattern, together with the frozen variables, generates~$\Rcal$. 
To prove the reverse inclusion 
$\AA(\tilde\xx,\tilde B)\subset \Rcal$,
we need to show that each cluster variable in the seed
pattern generated by $(\tilde\xx,\tilde B)$ is an element of~$\Rcal$ 
rather than merely a rational function in~$\FFcal$.
In this section, we give sufficient conditions that guarantee the
latter inclusion. 
%We use \cite{eisenbud} as a standard commutative algebra reference. 

Recall that $\Rcal$ is \emph{normal} if it is
integrally closed in~$\FFcal$. 
This property is in particular satisfied if $\Rcal$ is \emph{factorial} 
(or a \emph{unique factorization domain}). 
%i.e., if any nonzero element of $\Rcal$ is uniquely factored 
%(up to reordering and multiplication by units) 
%as a product of irreducible elements. 

Recall that $\Rcal$ is called \emph{Noetherian} if any ascending chain of
ideals stabilizes. %has a maximal element.
This is in particular satisfied if $\Rcal$ is \emph{finitely generated}
(over~$\CC$). 

All rings of interest to us will be factorial and finitely generated,
hence normal and Noetherian. 

%Let $\Rcal$ be a finitely generated $\CC$-algebra. 
Let us call two elements $r,r'\in \Rcal$ \emph{coprime} 
if they are not contained in the same prime ideal of height~1. 
%if the ideal they generate has height at least~$2$. % in~$\Rcal$. 
%If $\Rcal$ is a finitely generated $\CC$-algebra, then this is equivalent
%to saying that the locus of their common zeros 
%has codimension~$\ge 2$ in $\operatorname{Spec}(\Rcal)$.
If~$\Rcal$ is factorial, then such ideals are principal, 
and one recovers the usual definition
of coprimality 
($r$~and $r'$ are coprime if $\operatorname{gcd}(r,r')$ is a unit). 

%SF: probably don't need this:
%Indeed, an irreducible element in a factorial ring generates a prime ideal,
%so the vanishing loci in $\operatorname{Spec}(\Rcal)$ for $r$ and~$r'$ 
%are distinct irreducible subvarieties, so the codimension of their
%intersection is~$2$. 

\begin{proposition}[{``Starfish lemma''}]
\label{prop:cluster-criterion}
Let $\Rcal$ be a $\CC$-algebra and a %finitely generated $\CC$-algebra and a normal domain.
normal Noetherian domain.
Let~$(\tilde\xx,\tilde B)$ be a seed of rank $n$  
in the fraction field $\FFcal$
with $\tilde\xx = (x_1,\dots, x_m)$ for $n \leq m$
such that
\begin{enumerate}[leftmargin=.3in]
\item[{\rm (1)}]
\label{item:starfish-1}
all elements of $\tilde\xx$ belong to~$\Rcal$; 
\item[{\rm (2)}]
\label{item:starfish-2}
the cluster variables in $\tilde\xx$ are pairwise coprime; % (in~$\Rcal$); 
\item[{\rm (3)}]
\label{item:starfish-3}
for each cluster variable $x_k\in\tilde\xx$, the seed mutation~$\mu_k$
replaces $x_k$ with an element~$x_k'$ 
(cf.~\eqref{eq:exch-rel-geom})
that lies in $\Rcal$ and is coprime to~$x_k$. 
\end{enumerate}
%(Here we call two elements of $\Rcal$ coprime if the locus of their common zeros 
%has codimension~$\ge 2$ in $\operatorname{Spec}(\Rcal)$.)
Then $\AA(\tilde\xx,\tilde B)\subset \Rcal$.
%
%If, in addition, $\Rcal$ has a set of generators each of which 
%appears in the seeds mutation-equivalent to $(\tilde\xx,\tilde B)$, then 
%$\Rcal=\AA(\tilde\xx,\tilde B)$. 
\end{proposition}

We will give two proofs of 
Proposition~\ref{prop:cluster-criterion}, one using commutative algebra,
and one using algebraic geometry.  The commutative 
algebra proof relies on two lemmas. 

For~$P$ a prime ideal in~$\Rcal$, let $\Rcal_P=\Rcal[(\Rcal\setminus P)^{-1}]$ denote the localization
of~$\Rcal$ at $\Rcal\setminus P$. 

\begin{lemma}%[{\cite[Corollary~11.4]{eisenbud}}]
[{\cite[Theorem~11.5]{matsumura}}]
\label{lem:height-1}
For a normal Noetherian domain~$\Rcal$, 
the natural inclusion
$\Rcal\subset \bigcap_{\operatorname{ht} P=1} \Rcal_P$ (intersection over prime ideals~$P$ of height~$1$)
is an equality. 
\end{lemma}

\begin{lemma}
\label{lem:not-in-P}
Let $P$ be a height~$1$ prime ideal in~$\Rcal$. 
Then at least one~of the $n+1$ products
\[
(x_1\cdots x_n), %\quad, x_k'\prod_{i\neq k} x_i \quad
                 %(k\in\{1,\dots,n\})
(x_1'x_2\cdots x_n),
%(x_1x_2'\cdots x_n),
\dots,
(x_1\cdots x_{n-1} x_n')
\]
does not belong~to~$P$. 
\end{lemma}

\begin{proof}
Suppose that $x_1\cdots x_n\in P$. 
Since $P$ is prime, we have $x_k\in P$ for some~$k\le n$. 
Since $\operatorname{ht} P=1$,
the coprimality assumption~\eqref{item:starfish-2} implies
that $x_j\notin P$ for $j\in\{1,\dots,n\}-\{k\}$.
Similarly, \eqref{item:starfish-3} implies that $x_k'\notin P$. 
Again using that $P$ is prime, we conclude that 
$x_1\cdots x_k'\cdots x_n\notin P$. 
\end{proof}

\begin{proof}[Algebraic proof of Proposition~\ref{prop:cluster-criterion}]
We need to prove that each cluster variable~$z$ 
from any seed %$(\bar Q,\bar \tilde\xx)$ 
mutation equivalent to $(\tilde\xx,\tilde B)$ belongs to~$\Rcal$. 
By Lemma~\ref{lem:height-1}, it suffices to show that $z\in \Rcal_P$ for each prime
ideal~$P$ of height~$1$. 
%Let $P$ be such an ideal.
By Lemma~\ref{lem:not-in-P}, for any height $1$ prime $P$ in $\Rcal$,
there exists a cluster~$\xx'$ 
%(either $\xx$ or $\xx\cup\{x_k'\}-\{x_k\}$) 
such that $\prod_{x\in\xx'} x\in \Rcal\setminus P$. 
By~the Laurent Phenomenon
(Theorems~\ref{thm:Laurent} and~\ref{th:Laurent-sharper}), 
the cluster variable $z$ can be expressed 
as a Laurent polynomial in the elements of~$\xx'$,
with coefficients in $\CC[x_{n+1},\dots,x_m]$. 
Thus $z\in \Rcal[(\Rcal \setminus P)^{-1}]=\Rcal_P$, as desired. 
\end{proof}

%We now give a second proof of the Starfish Lemma which 
%uses algebraic geometry instead of commutative algebra.
%To prove that $\Rcal\supset\AA(\tilde\xx,\tilde B)$,
%We need to show that each cluster
%variable~$z$ from any seed %$(\bar Q,\bar \tilde\xx)$ 
%mutation equivalent to $(\tilde\xx,\tilde B)$ 
%belongs to~$\Rcal$. 
%The proof uses some basic algebraic geometry.

\begin{proof}[Geometric proof of Proposition~\ref{prop:cluster-criterion}]
Our assumptions~on the ring $\Rcal$ mean that 
it can be identified with the coordinate ring of an
(irreducible) normal affine complex 
algebraic variety $X\!=\!\operatorname{Spec}(\Rcal)$. 
Then the field of fractions of~$\Rcal$ is $\operatorname{Frac}(\Rcal)\!=\!\CC(X)$, the field of rational functions on~$X$. 
We need to show that each cluster variable~$z$ 
from any seed %$(\bar Q,\bar \tilde\xx)$ 
mutation equivalent to $(\tilde\xx,\tilde B)$ 
belongs to~$\Rcal$. 
%as above is not just a rational function in $\CC(X)$ but a regular
%function in $\CC[X]\!=\!\Rcal$. 
The key property that we need is the algebraic version of
\emph{Hartogs' continuation principle} for normal varieties 
(see, e.g., \cite[Chapter~2, 7.1]{danilov}) which asserts that a
function on~$X$ that is regular outside a closed algebraic subset
of codimension~$\ge 2$ is in fact regular everywhere on~$X$. 

%Let $Y\subset X$ be the locus where at least two of the cluster
%variables in $\tilde\xx$ vanish, or else one of them, say~$x_k$, vanishes
%together with the element~$x_k'$ defined
%by~\eqref{eq:exch-rel-geom}. 
Consider the subvariety
\[
	Y=\bigcup_{1 \leq i<j\leq n}\{x_i=x_j=0\} \cup \bigcup_{1 \leq k\leq n} \{x_k=x_k'=0\}\subset X. 
\]
The coprimeness conditions imposed on~$(\tilde\xx,\tilde B)$ 
imply that $\operatorname{codim}(Y)\ge 2$.
By the algebraic Hartogs' principle mentioned above, 
it now suffices to show that $z$ is regular on~$X\setminus Y$. 

\pagebreak[3]
 
The complement $X\setminus Y$ consists of the points $x\in X$ such
that
\begin{itemize}[leftmargin=.2in]
\item
at most one of the cluster variables in $\tilde\xx$ vanishes at~$x$, and
\item
for each pair $(x_k,x_k')$ as above, either $x_k$ or $x_k'$ does not vanish
at~$x$. 
\end{itemize}
Hence there is a seed $(Q',\tilde\xx')$ 
(either the original seed $(\tilde\xx,\tilde B)$ or one of the adjacent seeds 
$\mu_k(\tilde\xx,\tilde B)$) none of whose cluster variables vanishes at~$x$;  
moreover $\tilde\xx'\subset\CC[X]$. 
Then the Laurent Phenomenon (Theorems~\ref{thm:Laurent}
and~\ref{th:Laurent-sharper}) implies that 
our distant cluster variable $z$ is regular at~$x$, as desired. 
%By the algebraic Hartogs' principle mentioned above, 
%it follows that $z$ is regular on~$X$, and
%we are done. 
%\fi
\end{proof}

\begin{remark}
The arguments given above actually establish a stronger statement:
%(not needed elsewhere in the paper):
under the conditions of Proposition~\ref{prop:cluster-criterion},
the ring~$\Rcal$ contains the \emph{upper cluster algebra}
associated with $\AA(\tilde\xx,\tilde B)$ (see~\cite{ca3}),
or more precisely the subalgebra of  $\FFcal$ consisting of the elements
which, when expressed in terms of any extended cluster,
are Laurent polynomials in the cluster variables and 
ordinary polynomials in the coefficient variables. 
%The final conclusion of Proposition~\ref{prop:cluster-criterion} 
%can thus be restated as $\Rcal=\AA(\tilde\xx,\tilde B)=\overline\AA(\tilde\xx,\tilde B)$. 
%In particular, the conditions of
%Proposition~\ref{prop:cluster-criterion} imply that in this case,
%the upper cluster algebra coincides with the ordinary cluster algebra
%(localized at coefficient variables). 
\end{remark}

\begin{remark}
The versions of the Starfish Lemma and the Laurent phenomenon given in~\cite{ca3} (implicit) 
and~\cite{Gross-Hacking-Keel} 
are predicated on invertibility of coefficient variables 
(that is, the ground ring is the ring of Laurent polynomials in the coefficient variables)
and concern the upper cluster algebra.
%LW: deleted the following phrase since we already introduced upper cluster
% algebra in previous remark.
%, a closely related but different construct. 
\end{remark}

%SF: I am not sure if we need the remark below. I first wrote it, then
%removed...
\iffalse
\begin{remark}
In the future, we will need to supplement
Proposition~\ref{prop:cluster-criterion} 
by an argument establishing the inclusion $\AA(\tilde\xx,\tilde
B)\supset \Rcal$.
This type of argument typically relies on the knowledge of a (finite) set~$S$ of
generators for~$\Rcal$, which we then show to be contained in the set of
cluster and coefficient variables for $\AA(\tilde\xx,\tilde B)\supset
\Rcal$.
Thus in these settings, it is known from the outset that $\Rcal$ is
finitely generated (hence Noetherian). 
\end{remark}
\fi

\begin{corollary}
\label{cor:cluster-criterion}
Let $\Rcal$ be a finitely generated 
factorial $\CC$-algebra.
Let $(\tilde\xx,\tilde B)$ be a seed in the quotient field of $\Rcal$ 
such that 
all cluster variables of $\tilde\xx$ and all elements of clusters adjacent to~$\tilde\xx$ 
are irreducible elements of~$\Rcal$. 
Then $\AA(\tilde\xx,\tilde B)\subset \Rcal$. 
%
%If, in addition, the union of extended clusters in $\AA(\tilde\xx,\tilde B)$ 
%contains a generating set for~$\Rcal$, then $\Rcal=\AA(\tilde\xx,\tilde B)$. 
\end{corollary}

\begin{proof}
The only conditions in 
Proposition~\ref{prop:cluster-criterion} that we need to check
are the ones concerning coprimality. 
Two elements of~$\tilde\xx$
cannot differ by a scalar factor since they are algebraically independent. 
Similarly, if $x_k$ and $x_k'$ were to differ
by a scalar factor, then the exchange relation~\eqref{eq:exch-rel-geom} would give 
an algebraic dependence in~$\tilde\xx$. 
\end{proof}

Suppose that a $\CC$-algebra~$\Rcal$ satisfies the conditions in the first sentence of 
Proposition~\ref{prop:cluster-criterion}
(or Corollary~\ref{cor:cluster-criterion}). 
In order to identify a cluster structure in~$\Rcal$, 
it suffices to exhibit a seed $(\tilde\xx,\tilde B)$ such that
\begin{itemize}[leftmargin=.3in]
\item[(i)]
all cluster variables
 of $\tilde\xx$ and of the clusters adjacent to~$\tilde\xx$ 
are irreducible elements of~$\Rcal$; 
\item[(ii)]
the seed pattern generated by $(\tilde\xx,\tilde B)$ contains a
generating set for~$\Rcal$. 
\end{itemize}
This is however easier said than done. 

Regarding condition~(ii) above, let us make the following simple observation. 

\begin{proposition}
\label{pr:fin-gen-by-cl-var}
Let $\Rcal$ be a cluster algebra that is finitely generated (over~$\CC$) as a $\CC$-algebra.  
Then $\Rcal$ is generated by a finite subset~of 
cluster and coefficient variables. 
\end{proposition}

\begin{proof}
Let $S$ be a finite generating set for~$\Rcal$, 
and let $\mathcal{X}$ be the set of cluster and coefficient variables. 
Since each $s\in S$ can be written as a polynomial in the elements of a finite
subset $\mathcal{X}_s\!\subset\! \mathcal{X}$, 
we conclude that the finite set $\bigcup_{s\in S}\mathcal{X}_s\subset\Xcal$ generates~$\Rcal$. 
\end{proof}

We next review some general algebraic criteria that can be used 
to check that a given $\CC$-algebra~$\Rcal$ satisfies the conditions in 
Proposition~\ref{prop:cluster-criterion} 
or Corollary~\ref{cor:cluster-criterion}.

The fact that $\Rcal$ is a domain will usually be immediate,
e.g.\ when $\Rcal$ is given as a subring of a polynomial ring. 

Perhaps the most famous result concerning finite generation 
is (the modern version of) Hilbert's Theorem, 
see, e.g., \cite[Theorem~3.5]{popov-vinberg}: 

\begin{theorem} %[{\emph{Hilbert's Theorem}; see, e.g.,
                %\cite[Theorem~3.5]{popov-vinberg}}]
\label{th:hilbert}
Let $G$ be a reductive algebraic group acting on an affine algebraic variety~$X$. 
Then the ring of invariants $\CC[X]^G$ is finitely generated. 
\end{theorem}

For the purposes of applying Proposition~\ref{prop:cluster-criterion},
the following version is particularly useful.

\begin{theorem}[{\cite[Proposition~3.1]{dolgachev}}]
\label{th:hilbert-dolgachev}
Let $G$ be a reductive algebraic group acting algebraically on a
normal finitely generated $\CC$-algebra~$A$. 
Then $A^G$ is a normal finitely generated $\CC$-algebra.
\end{theorem}

Even when a group is not reductive, the ring of its invariants may be
finitely generated.
The most important case to us is the following. 

\begin{theorem}[{\cite[Theorem~5.4]{dolgachev}}]
\label{th:A^U-fin-generated}
Let $G$ be a reductive group 
acting rationally on a finitely generated $\CC$-algebra~$A$. 
Let $U$ be a maximal unipotent group of~$G$. 
Then the subalgebra $A^U$ of $U$-invariant elements in~$A$ is finitely
generated over~$\CC$. 
\end{theorem}

We conclude this section with a couple of factoriality criteria,
see \cite[Theorem~3.17]{popov-vinberg},
\cite{willenbring-zuckerman} and references therein. 

\begin{proposition} %[{see, e.g., \cite{willenbring-zuckerman} and
                    %references therein}]
\label{pr:voskresenskii}
Let $G$ be a connected, simply connected
semi\-simple complex Lie group.
Then the ring of regular functions $\CC[G]$ is factorial. 
\end{proposition}

\begin{theorem} %[{\cite[Theorem~3.17]{popov-vinberg}}]
\label{th:inv-ring-is-factorial}
Let $G$ be a connected algebraic group 
acting on an affine algebraic variety~$X$. 
If $G$ has no nontrivial characters and $\CC[X]$ is factorial, then so is $\CC[X]^G$. 
\end{theorem}

\begin{remark}
The paper~\cite{gls-factorial} provides factoriality criteria for cluster algebras. 
It also shows that a cluster algebra contains no nontrivial units,
and all cluster variables are irreducible elements. 
\end{remark}

\newpage

\section{Cluster structure in the ring $\CC[\SL_k]^U$}
\label{sec:rings-baseaffine}

In this section, we identify a cluster algebra structure in the ring $\CC[\SL_k]^U$, 
the coordinate ring of the basic affine space for the special linear
group. This ring made its first appearance in
Section~\ref{sec:baseaffine}. 
%The essential ingredients of this construction are by now familiar to
%the reader; we review them below for the sake of completeness. 
We start by reviewing the key features of this construction. 

Let $V\cong\CC^k$ be a $k$-dimensional complex vector space.
After choosing a basis in~$V$, we can identify $\SL_k$ with 
the special linear group $G=\SL(V)$ of complex matrices with determinant~1.   
The subgroup $U\subset G$ of unipotent lower-triangular matrices
acts on~$G$ by left multiplication. 
This action induces the action of~$U$ on the coordinate ring~$\CC[G]$.
We will show that the ring $\CC[G]^U$ 
of $U$-invariant regular functions on~$G$ has a natural structure of a
cluster algebra. 

We note that $U$ is not reductive, so Theorem~\ref{th:hilbert} does not
apply. 
Still, $\CC[G]^U$ is finitely generated by
Theorem~\ref{th:A^U-fin-generated}. 
%SF: This is a special case of \cite[Theorem~3.13]{popov-vinberg}. Do
%we want to mention this?
As mentioned in Section~\ref{sec:baseaffine}, 
this can be made explicit as follows. 
Recall that a \emph{flag minor} $P_J$ of a $k\times k$ matrix~$z$ 
(here $J\subset\{1,\dots,k\}$) 
is the determinant of the submatrix of~$z$ occupying the columns
labeled by~$J$ and the rows labeled $1, 2, \dots,|J|$. 

\begin{theorem}
\label{th:generating-G/U}
The ring of invariants $\CC[\SL_k]^U$ 
is generated by the $2^k-2$ \emph{flag minors}~$P_J$; 
here $J$ runs over nonempty proper subsets of~$\{1,\dots,k\}$. 
\end{theorem}

%As in the case of Pl\"ucker coordinates on a Grassmannian, 
The ideal of relations satisfied by the flag minors is generated by
certain generalized Grassmann-Pl\"ucker relations
(which we will not rely upon). 

%\begin{remark}
%\label{rem:G/N-general}
Theorem~\ref{th:generating-G/U} is a consequence of the classical
construction of irreducible representations of special linear groups,
see, e.g.,~\cite{fulton-harris, procesi-textbook}. 
This construction generalizes to an arbitrary connected, simply connected
semi\-simple complex Lie group~$G$ and its maximal unipotent
subgroup~$U$.
The role of flag minors is played by certain matrix elements in
fundamental representations of~$G$. 
See the end of this section for additional details. 
%\end{remark}

\begin{corollary}
The ring $\CC[\SL_k]^U$ is factorial.
\end{corollary}

\begin{proof}
This follows from Proposition~\ref{pr:voskresenskii} and 
Theorem~\ref{th:inv-ring-is-factorial}. 
(The polynomial ring $\CC[U]$ has no nontrivial units.) 
\end{proof}

%We now explain how to associate special seeds to the base affine space.  

\begin{definition}
%[\emph{Labeled seeds for $\base$ from wiring diagrams}]
\label{def:seedwd}
Let $D$ be a wiring diagram with $k$ strands.  
%as introduced in Section~\ref{sec:baseaffine}.
Let $\FFcal$  be the field of rational functions in the 
chamber minors of~$D$, cf.\ Section~\ref{sec:baseaffine}. 
We associate to $D$ the 
pair $(\tilde \xx(D), \tilde B(D))$, where
\begin{itemize}[leftmargin=.2in]
\item $\tilde \xx(D)$ consists of the {chamber minors}
of~$D$, listed so that the minors indexed by the bounded chambers precede the 
minors indexed by the unbounded ones;
% = \{x_1, \dots, x_m\}$ is an
%algebraically independent $m$-tuple of elements of~$\FFcal$ forming a
%\emph{free generating set}; that is, $x_1, \dots, x_m$
%are algebraically independent, and $\FFcal = \QQ(x_1, \dots, x_m)$;
\item $\tilde B(D)%=(b_{ij})
$ 
is the signed adjacency matrix of the quiver $Q(D)$
from Definition~\ref{def:quiverwd}.
\end{itemize}
Note that $\tilde B(D)$ is 
a $\frac{(k-1)(k+2)}{2} \times \frac{(k-1)(k-2)}{2}$ 
integer matrix whose rows are indexed by all chambers
and whose columns are indexed by the bounded chambers.
The minors corresponding to the bounded (resp., unbounded) 
chambers are the cluster variables (resp., frozen variables) of this seed.
\end{definition}

%\Comment{Lots of problems!  Move 6.5.8 and 6.5.9 to Section 3.2 -- making
%sure that the notation is defined there --  alternatively put just 6.5.9 there -- and also 
%move Example 6.5.4 (specifying which ring we're talking about), so that it comes after Thoerem 6.5.5.  Wherever we move things, put SL3, then SL4, then SL5
%in that order. So maybe move SL3 to rank 1 section, then leave SL4 where it is,
%then move SL5 after theorem.}

%\Comment{see Lauren's email dated 12/21/14, subject: Base affine $SL_5 - D_6$}

\begin{theorem}
Let $G = \SL_k(\CC)$. 
\begin{enumerate}
\item
\label{item:SLk/U-1}
For any wiring diagram with $k$ strands, 
the pair $(\tilde \xx(D), \tilde B(D))$ is a seed in the
  field of fractions for $\CC[G]^U$. 

\item
\label{item:SLk/U-2}
All seeds $(\tilde \xx(D), \tilde B(D))$ are mutation equivalent
to each other.

\item
\label{item:SLk/U-3}
The seed pattern containing the seeds $(\tilde \xx(D), \tilde B(D))$
defines a cluster algebra structure in~$\CC[G]^U$. 
That is, $\AA(\tilde \xx(D), \tilde B(D))=\CC[G]^U$. 
\end{enumerate}
\end{theorem}

\begin{proof}
Statement~\eqref{item:SLk/U-2} is part of 
Exercise~\ref{exercise:quiverexchange}. 
%	All wiring diagrams are connected via flips, 
%and each flag minor appears as a chamber minor in one of them. 
We then conclude that any flag minor can be expressed 
as a rational function in the elements of a given extended cluster~$\tilde\xx(D)$. 
Since 
$|\tilde\xx(D)|=\frac{(k-1)(k+2)}{2}=\dim(U\backslash G)$, 
it follows that the elements of $\tilde\xx(D)$ are algebraically independent, proving
statement~\eqref{item:SLk/U-1}. 

It remains to prove statement~\eqref{item:SLk/U-3}.
Since each flag minor appears in~some  extended
cluster~$\tilde\xx(D)$, Theorem~\ref{th:generating-G/U} implies 
that $\AA(\tilde \xx(D), \tilde B(D))\supset\CC[G]^U$. 

We prove the inclusion $\AA(\tilde \xx(D), \tilde
B(D))\subset\CC[G]^U$
using Proposition~\ref{prop:cluster-criterion}. 
Let us choose the seed associated to the wiring diagram~$D$ 
of the kind shown in Figure~\ref{fig:special-wiring}. 
The quiver is shown in Figure \ref{fig:special-wiring-seed}.

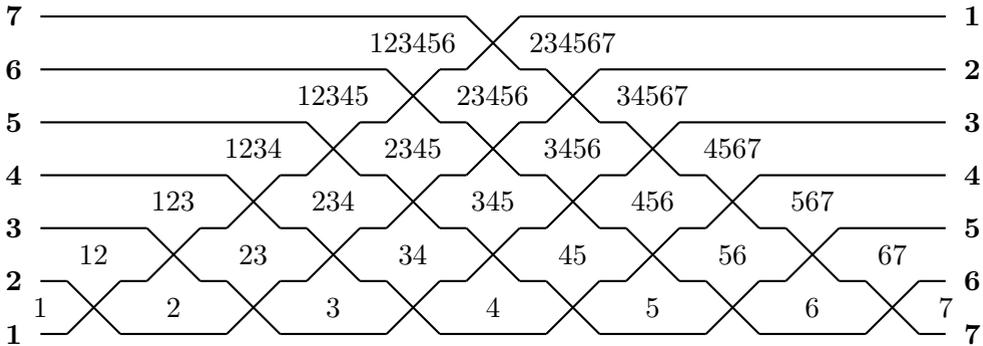
\begin{figure}[ht]
\setlength{\unitlength}{1pt}
\begin{center}
\begin{picture}(340,120)(0,5)
\thicklines
\dark{

  \put(0,0){\line(1,0){10}}
  \multiput(30,0)(60,0){5}{\line(1,0){40}}
  \put(330,0){\line(1,0){10}}

  \put(0,20){\line(1,0){10}}
  \multiput(30,20)(30,0){10}{\line(1,0){10}}
  \put(330,20){\line(1,0){10}}

  \put(0,40){\line(1,0){40}}
  \multiput(60,40)(30,0){8}{\line(1,0){10}}
  \put(300,40){\line(1,0){40}}

  \put(0,60){\line(1,0){70}}
  \multiput(90,60)(30,0){6}{\line(1,0){10}}
  \put(270,60){\line(1,0){70}}

  \put(0,80){\line(1,0){100}}
  \multiput(120,80)(30,0){4}{\line(1,0){10}}
  \put(240,80){\line(1,0){100}}

  \put(0,100){\line(1,0){130}}
  \multiput(150,100)(30,0){2}{\line(1,0){10}}
  \put(210,100){\line(1,0){130}}

  \put(0,120){\line(1,0){160}}
  \put(180,120){\line(1,0){160}}

  \multiput(10,0)(60,0){6}{\line(1,1){20}}
  \multiput(10,20)(60,0){6}{\line(1,-1){20}}
  \multiput(40,20)(60,0){5}{\line(1,1){20}}
  \multiput(40,40)(60,0){5}{\line(1,-1){20}}
  \multiput(70,40)(60,0){4}{\line(1,1){20}}
  \multiput(70,60)(60,0){4}{\line(1,-1){20}}
  \multiput(100,60)(60,0){3}{\line(1,1){20}}
  \multiput(100,80)(60,0){3}{\line(1,-1){20}}
  \multiput(130,80)(60,0){2}{\line(1,1){20}}
  \multiput(130,100)(60,0){2}{\line(1,-1){20}}
  \multiput(160,100)(60,0){1}{\line(1,1){20}}
  \multiput(160,120)(60,0){1}{\line(1,-1){20}}
}

  \put(-10,0){\makebox(0,0){$\mathbf{1}$}}
  \put(-10,20){\makebox(0,0){$\mathbf{2}$}}
  \put(-10,40){\makebox(0,0){$\mathbf{3}$}}
  \put(-10,60){\makebox(0,0){$\mathbf{4}$}}
  \put(-10,80){\makebox(0,0){$\mathbf{5}$}}
  \put(-10,100){\makebox(0,0){$\mathbf{6}$}}
  \put(-10,120){\makebox(0,0){$\mathbf{7}$}}

  \put(350,0){\makebox(0,0){$\mathbf{7}$}}
  \put(350,20){\makebox(0,0){$\mathbf{6}$}}
  \put(350,40){\makebox(0,0){$\mathbf{5}$}}
  \put(350,60){\makebox(0,0){$\mathbf{4}$}}
  \put(350,80){\makebox(0,0){$\mathbf{3}$}}
  \put(350,100){\makebox(0,0){$\mathbf{2}$}}
  \put(350,120){\makebox(0,0){$\mathbf{1}$}}

  \put(0,10){\makebox(0,0){$1$}}
  \put(50,10){\makebox(0,0){$2$}}
  \put(110,10){\makebox(0,0){$3$}}
  \put(170,10){\makebox(0,0){$4$}}
  \put(230,10){\makebox(0,0){$5$}}
  \put(290,10){\makebox(0,0){$6$}}
  \put(340,10){\makebox(0,0){$7$}}

  \put(20,30){\makebox(0,0){$12$}}
  \put(80,30){\makebox(0,0){$23$}}
  \put(140,30){\makebox(0,0){$34$}}
  \put(200,30){\makebox(0,0){$45$}}
  \put(260,30){\makebox(0,0){$56$}}
  \put(320,30){\makebox(0,0){$67$}}

  \put(50,50){\makebox(0,0){$123$}}
  \put(110,50){\makebox(0,0){$234$}}
  \put(170,50){\makebox(0,0){$345$}}
  \put(230,50){\makebox(0,0){$456$}}
  \put(290,50){\makebox(0,0){$567$}}

  \put(80,70){\makebox(0,0){$1234$}}
  \put(140,70){\makebox(0,0){$2345$}}
  \put(200,70){\makebox(0,0){$3456$}}
  \put(260,70){\makebox(0,0){$4567$}}

  \put(110,90){\makebox(0,0){$12345$}}
  \put(170,90){\makebox(0,0){$23456$}}
  \put(230,90){\makebox(0,0){$34567$}}

  \put(140,110){\makebox(0,0){$123456$}}
  \put(200,110){\makebox(0,0){$234567$}}

\end{picture}
\end{center}

\caption{A special wiring diagram for $k=7$, and its chamber minors.}
\label{fig:special-wiring}
\end{figure}

%\put(4,10){\vector(1,0){12}}

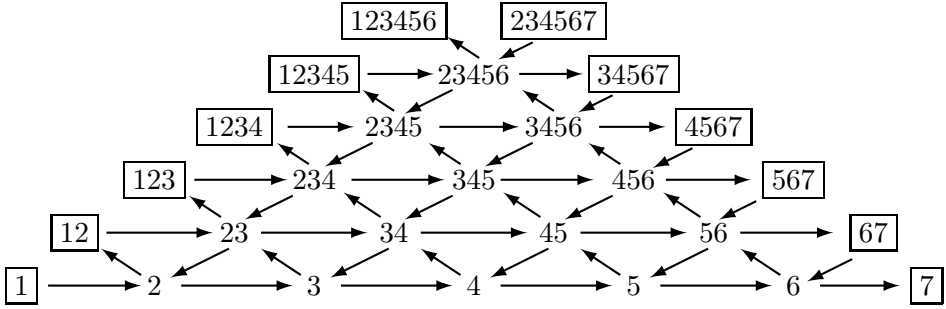
\begin{figure}[ht]
\setlength{\unitlength}{1pt}
\begin{center}
\begin{picture}(340,105)(0,10)
\thicklines
\dark{

\put(10,10){\vector(1,0){35}}
\put(60,10){\vector(1,0){42}}
\put(120,10){\vector(1,0){42}}
\put(180,10){\vector(1,0){42}}
\put(240,10){\vector(1,0){42}}
\put(300,10){\vector(1,0){30}}

\put(32,30){\vector(1,0){40}}
\put(90,30){\vector(1,0){40}}
\put(150,30){\vector(1,0){40}}
\put(210,30){\vector(1,0){40}}
\put(270,30){\vector(1,0){35}}

\put(65,50){\vector(1,0){35}}
\put(124,50){\vector(1,0){35}}
\put(180,50){\vector(1,0){35}}
\put(242,50){\vector(1,0){32}}

\put(100,70){\vector(1,0){25}}
\put(157,70){\vector(1,0){30}}
\put(217,70){\vector(1,0){25}}

\put(130,90){\vector(1,0){25}}
\put(187,90){\vector(1,0){24}}

\put(45,14){\vector(-3,2){15}}
\put(105,14){\vector(-3,2){15}}
\put(165,14){\vector(-3,2){15}}
\put(225,14){\vector(-3,2){15}}
\put(285,14){\vector(-3,2){15}}

\put(75,36){\vector(-3,2){12}}
\put(135,36){\vector(-3,2){14}}
\put(195,36){\vector(-3,2){14}}
\put(255,36){\vector(-3,2){14}}

\put(108,56){\vector(-3,2){12}}
\put(165,56){\vector(-3,2){12}}
\put(225,56){\vector(-3,2){12}}

\put(140,76){\vector(-3,2){12}}
\put(200,76){\vector(-3,2){12}}

\put(172,96){\vector(-3,2){12}}

\put(78,24){\vector(-2,-1){22}}
\put(138,24){\vector(-2,-1){22}}
\put(198,24){\vector(-2,-1){22}}
\put(258,24){\vector(-2,-1){22}}
\put(312,20){\vector(-2,-1){17}}

\put(102,44){\vector(-2,-1){18}}
\put(162,44){\vector(-2,-1){18}}
\put(222,44){\vector(-2,-1){18}}
\put(276,42){\vector(-2,-1){14}}

\put(132,64){\vector(-2,-1){18}}
\put(192,64){\vector(-2,-1){18}}
\put(252,62){\vector(-2,-1){15}}

\put(162,84){\vector(-2,-1){18}}
\put(222,82){\vector(-2,-1){13}}

\put(192,102){\vector(-2,-1){13}}

%  \put(0,0){\line(1,0){10}}
}

  \put(0,10){\makebox(0,0){$\boxed{1}$}}
  \put(50,10){\makebox(0,0){$2$}}
  \put(110,10){\makebox(0,0){$3$}}
  \put(170,10){\makebox(0,0){$4$}}
  \put(230,10){\makebox(0,0){$5$}}
  \put(290,10){\makebox(0,0){$6$}}
  \put(340,10){\makebox(0,0){$\boxed{7}$}}

  \put(20,30){\makebox(0,0){$\boxed{12}$}}
  \put(80,30){\makebox(0,0){$23$}}
  \put(140,30){\makebox(0,0){$34$}}
  \put(200,30){\makebox(0,0){$45$}}
  \put(260,30){\makebox(0,0){$56$}}
  \put(320,30){\makebox(0,0){$\boxed{67}$}}

  \put(50,50){\makebox(0,0){$\boxed{123}$}}
  \put(110,50){\makebox(0,0){$234$}}
  \put(170,50){\makebox(0,0){$345$}}
  \put(230,50){\makebox(0,0){$456$}}
  \put(290,50){\makebox(0,0){$\boxed{567}$}}

  \put(80,70){\makebox(0,0){$\boxed{1234}$}}
  \put(140,70){\makebox(0,0){$2345$}}
  \put(200,70){\makebox(0,0){$3456$}}
  \put(260,70){\makebox(0,0){$\boxed{4567}$}}

  \put(110,90){\makebox(0,0){$\boxed{12345}$}}
  \put(170,90){\makebox(0,0){$23456$}}
  \put(230,90){\makebox(0,0){$\boxed{34567}$}}

  \put(140,110){\makebox(0,0){$\boxed{123456}$}}
  \put(200,110){\makebox(0,0){$\boxed{234567}$}}

\end{picture}
\end{center}

\caption{The seed corresponding to the wiring diagram in Figure \ref{fig:special-wiring}.}
\label{fig:special-wiring-seed}
\end{figure}

All the elements of $\tilde\xx(D)$ are flag minors, so they belong
to~$\CC[G]^U$.
Moreover they are irreducible polynomials, hence irreducible elements
of~$\CC[G]^U$. 
This follows from the well-known fact that 
the determinant of a matrix of indeterminates
is an irreducible polynomial, see~\cite[Theorem~3.2]{denis-serre}. 

Let us compute the elements of the clusters adjacent
to~$\tilde\xx(D)$. 
Note that the chamber minors of $D$ are \emph{solid},
i.e., they have column sets of the form
\[
[a,d]=\{a,a+1,\dots,d-1,d\},
\]
for $1\le a\le d\le k$. 
Note that mutations at vertices in the bottom row of the quiver
can be understood using the braid moves we studied earlier, 
\emph{cf.\ } Figure~\ref{fig:moves1}.  To understand mutations 
at the other vertices of the quiver, note that 
a typical chamber minor $P_{[b,c]}\in\tilde \xx(D)$ 
is exchanged with the element $\Omega\in\operatorname{Frac}(\CC[G]^U)$ given by
\begin{equation}
\label{eq:Omega-Jbc}
\Omega=\dfrac{P_{Jabc}\,P_{Jcd}\,P_{Jb}+P_{Jab}\,P_{Jbcd}\,P_{Jc}}{P_{Jbc}},
\end{equation}
where we used the shorthand
\begin{align}
\nonumber
a&=b-1,\\
\nonumber
d&=c+1,\\
\nonumber
J&=[b+1,c-1],\\
\label{eq:Jbc}
Jbc&=J\cup\{b,c\}=[b,c], %\ \text{etc.}
\end{align}
and similarly for $Jb$, $Jc$, etc. 
This can be seen by examining the quiver $Q(D)$ in the vicinity of the
vertex associated with $P_{[b,c]}=P_{Jbc}$,
see Figure~\ref{fig:special-wiring-Jabcd}.

\begin{figure}[ht]
\setlength{\unitlength}{1pt}
\begin{center}
\begin{picture}(340,120)(0,5)
\thicklines
\dark{

  \put(0,0){\line(1,0){10}}
  \multiput(30,0)(60,0){5}{\line(1,0){40}}
  \put(330,0){\line(1,0){10}}

  \put(0,20){\line(1,0){10}}
  \multiput(30,20)(30,0){10}{\line(1,0){10}}
  \put(330,20){\line(1,0){10}}

  \put(0,40){\line(1,0){40}}
  \multiput(60,40)(30,0){8}{\line(1,0){10}}
  \put(300,40){\line(1,0){40}}

  \put(0,60){\line(1,0){70}}
  \multiput(90,60)(30,0){6}{\line(1,0){10}}
  \put(270,60){\line(1,0){70}}

  \put(0,80){\line(1,0){100}}
  \multiput(120,80)(30,0){4}{\line(1,0){10}}
  \put(240,80){\line(1,0){100}}

  \put(0,100){\line(1,0){130}}
  \multiput(150,100)(30,0){2}{\line(1,0){10}}
  \put(210,100){\line(1,0){130}}

  \put(0,120){\line(1,0){160}}
  \put(180,120){\line(1,0){160}}

  \multiput(10,0)(60,0){6}{\line(1,1){20}}
  \multiput(10,20)(60,0){6}{\line(1,-1){20}}
  \multiput(40,20)(60,0){5}{\line(1,1){20}}
  \multiput(40,40)(60,0){5}{\line(1,-1){20}}
  \multiput(70,40)(60,0){4}{\line(1,1){20}}
  \multiput(70,60)(60,0){4}{\line(1,-1){20}}
  \multiput(100,60)(60,0){3}{\line(1,1){20}}
  \multiput(100,80)(60,0){3}{\line(1,-1){20}}
  \multiput(130,80)(60,0){2}{\line(1,1){20}}
  \multiput(130,100)(60,0){2}{\line(1,-1){20}}
  \multiput(160,100)(60,0){1}{\line(1,1){20}}
  \multiput(160,120)(60,0){1}{\line(1,-1){20}}
}

  \put(-10,20){\makebox(0,0){$a$}}
  \put(-10,40){\makebox(0,0){$b$}}
  \put(-10,100){\makebox(0,0){$d$}}
  \put(-10,80){\makebox(0,0){$c$}}

  \put(350,40){\makebox(0,0){$c$}}
  \put(350,20){\makebox(0,0){$d$}}
  \put(350,80){\makebox(0,0){$b$}}
  \put(350,100){\makebox(0,0){$a$}}

  \put(140,30){\makebox(0,0){$Jb$}}
  \put(200,30){\makebox(0,0){$Jc$}}

  \put(110,50){\makebox(0,0){$Jab$}}
  \put(170,50){\makebox(0,0){$Jbc$}}
  \put(230,50){\makebox(0,0){$Jcd$}}

  \put(140,70){\makebox(0,0){$Jabc$}}
  \put(200,70){\makebox(0,0){$Jbcd$}}
\end{picture}
\setlength{\unitlength}{2.5pt} 
\begin{picture}(60,30)(0,10) 
\thinlines 

%\put( 0,10){\makebox(0,0){$\mathbf{_{1}}$}}
\put(20,10){\makebox(0,0){$Jb$}}
\put(40,10){\makebox(0,0){$Jc$}}
%\put(60,10){\makebox(0,0){$\mathbf{_{4}}$}}

\put(10,20){\makebox(0,0){$Jab$}}
\put(30,20){\makebox(0,0){$Jbc$}}
\put(50,20){\makebox(0,0){$Jcd$}}

\put(20,30){\makebox(0,0){$Jabc$}}
\put(42,30){\makebox(0,0){$Jbcd$}}

\thicklines 

%\put(4,10){\vector(1,0){12}}
%\put(24,10){\vector(1,0){12}}
%\put(44,10){\vector(1,0){12}}

\put(14,20){\vector(1,0){12}}
\put(34,20){\vector(1,0){12}}

\put(28,18){\vector(-1,-1){6}}
%\put(48,18){\vector(-1,-1){6}}
\put(38,28){\vector(-1,-1){6}}

%\put(17,12){\vector(-1,1){5}}
\put(37,12){\vector(-1,1){5}}
\put(27,22){\vector(-1,1){5}}

\end{picture}
\end{center}

\caption{Chamber minors appearing in the exchange relation for a flag
  minor $P_{Jbc}$ in a special wiring diagram,
and part of the associated quiver.
Only the arrows incident to the vertex $Jbc$ are shown.}
\label{fig:special-wiring-Jabcd}
\end{figure}
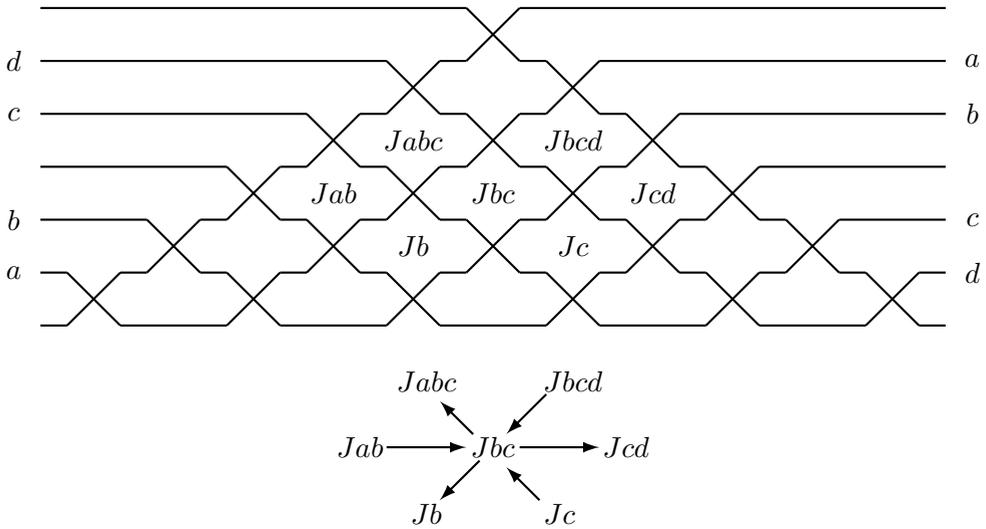

We verify that $\Omega\in\CC[G]^U$
by expressing $\Omega$ as a polynomial in flag minors: 

\begin{lemma}
The rational function $\Omega$ defined by~\eqref{eq:Omega-Jbc} satisfies
\begin{equation}
\label{eq:Omega-Jbc2}
\Omega %= \dfrac{P_{Jabc}\,P_{Jcd}\,P_{Jb}+P_{Jab}\,P_{Jbcd}\,P_{Jc}}{P_{Jbc}} 
= -P_{Ja} P_{Jbcd} + P_{Jb} P_{Jacd}.
\end{equation}
\end{lemma}

\begin{proof}
Follows from \eqref{eq:Omega}  by virtue of
Muir's Law %of extensible minors 
(Proposition~\ref{prop:Muir}).
\end{proof}

It remains to prove that $\Omega$ is coprime to~$P_{Jbc}$. 
Let $f$ denote the specialization that sets the top \hbox{$|J|+2$} ent\-ries in column~$a$ to zero. 
Note that $f$ leaves $P_{Jbc}$ unchanged, and $f(P_{Ja})=0$. 
We know that $P_{Jbc}$ is irreducible. 
If $\Omega$ were divisible by~$P_{Jbc}$, then
$f(\Omega)=\pm P_{Jb}\cdot z_{|J|+3,a}\cdot P_{Jcd}$  
would also be divisible by~$P_{Jbc}$.
But $f(\Omega)$ is a product of three
irreducible polynomials none of which is a scalar multiple of~$P_{Jbc}$. 
\end{proof}

\pagebreak[3]

The special cases $G=\SL_3$ and $G=\SL_4$ of this construction
have been presented in 
Examples~\ref{ex:base-affine-SL3} and~\ref{example:SL4},
respectively.   We now discuss the case $G = \SL_5$.
%Already in the case $G=\SL_4$ we saw that while each wiring diagram 
%encodes a seed for~$\CC[G]^N$, 
%in general not all seeds have this form.

\begin{example}\label{ex:SL5}
Let $G=\SL_5$.  
In this case, the cluster structure~in~$\CC[G]^U$
is of type~$D_6$, and accordingly has $36$ cluster variables, cf.\ Figure~\ref{tab:cluster-numbers}.  
There are $8$ coefficient variables, all of them flag minors. 

To compute all cluster variables, one can use the following method,
cf.\ \cite[Proposition~11.1(1)]{ca4}. 
Start with a seed coming from a wiring diagram. 
Apply mutations to obtain a seed whose quiver is a bipartite
orientation of the Dynkin diagram type~$D_6$. 
(That is, each vertex is either a source or a sink.) 
Then repeatedly alternate between mutating at all sources 
and mutating at all sinks until all $36$ cluster variables are computed.
After each mutation, one needs to represent the new cluster variable
as a regular function, not just a rational one.
%This can be helped by locating each function within the list  
%shown in \eqref{eq:g(1,2|3,4)...}--\eqref{eq:j(1,2|3,4,5)...} below. 
%(Note that this algorithm was used in the proof of 
%\cref{thm:E}.)
At the end of this process, one determines that each of the $22=2^5-2-8$ 
flag minors which is not a coefficient variable is a cluster variable. 
The remaining 14 cluster variables are described as follows. 
Define
\begin{align*}
g(a,b|c,d) &= -P_a P_{bcd}+P_b P_{acd},\\
g(a,b|c,d|J) &= -P_{Ja} P_{Jbcd} + P_{Jb} P_{Jacd}, \\
h(a,b,c|d,e) &= -P_{ab} P_{cde} + P_{ac} P_{bde},\\
j(a,b|c,d,e) &= -P_a P_{bcde} + P_b P_{acde},
\end{align*}
where we use the shorthand $bcd = \{b,c,d\}$, $Ja = J \cup \{a\}$, etc.
Note that $g(a,b|c,d|J)$ is precisely the regular function $\Omega$
 from \eqref{eq:Omega-Jbc} and~\eqref{eq:Omega-Jbc2}. 
With this notation, the $14$ cluster variables in question turn out to be:
\begin{align}
\label{eq:g(1,2|3,4)...}
&g(1,2|3,4), g(2,3|4,5), g(4,5|1,2), g(1,3|4,5), g(1,2|3,5), \\
&g(1,2|3,4|5), g(2,3|4,5|1), g(4,5|1,2|3), g(1,3|4,5|2), g(1,2|3,5|4), \\
&h(1,2,3|4,5), h(5,4,3|2,1), \\
\label{eq:j(1,2|3,4,5)...}
&j(1,2|3,4,5), j(5,4|3,2,1).
\end{align}
\end{example}

For $k\ge 6$, the cluster structure on $\CC[\SL_k]^U$ 
that we described above is of infinite type. 
See Table \ref{tab:clustertype}.

\begin{table}[ht]
\begin{center}
\begin{tabular}{| l | l |}
\hline
Ring & Cluster type \\
\hline 
\hline
& \\[-.4cm]
$\CC[\SL_3]^U$ & $A_1$ \\[.1cm] \hline
& \\[-.4cm]
$\CC[\SL_4]^U$ & $A_3$ \\[.1cm] \hline
& \\[-.4cm]
$\CC[\SL_5]^U$ & $D_6$ \\[.1cm] \hline
& \\[-.4cm]
$\CC[\SL_k]^U$ for $k \geq 6$ & infinite type \\ \hline
\end{tabular}
%\vspace{.4cm}
\end{center}
\caption{The type of the standard cluster structure on $\CC[\SL_k]^U$.}
\label{tab:clustertype}
\end{table}

\begin{remark}
Let $U^+\subset G = \SL_k$ be the subgroup of unipotent \emph{upper-triangular} matrices.  
(Recall that we denoted by $U$ the subgroup of unipotent \emph{lower-triangular}
matrices.)  
Let $\varphi:\CC[G]^U \to \CC[U^+]$ be the ring map defined by 
restricting $U$-invariant functions on~$G$ to the subgroup~$U^+$.
Since every matrix entry of an element of $U^+$
can be written as a flag minor, $\varphi$~is onto.
The map $\varphi$ can be used to transform a cluster structure on~$\CC[G]^U$
into a cluster structure on~$\CC[U^+]$. 
(This boils down to removing the coefficient variables corresponding to the leading 
principal minors~$P_{1,\cdots,j}$, as $P_{1,\cdots,j}(u)=1$ for $u\in U^+$.)  

More generally, the coordinate ring of 
%$\CC[U]$, where $U\subset G$ is 
a maximal unipotent subgroup of any Kac-Moody group
is a cluster algebra \cite{GLS-kacmoody}.  
%LW says: is the statement below really true in general type?
%Note that a maximal unipotent subgroup 
%is an affine space, so its coordinate ring
%is just a polynomial ring.
\end{remark}

In this section we have so far discussed the basic affine space
in the case of $G = \SL_k$.  We now 
give a quick review of basic affine spaces and their significance
for general semi\-simple Lie groups $G$. 
%The reader interested exclusively in the type~$A$ case 
%of special linear groups can proceed to the next section. 
An excellent introduction is given in \cite[Section~2.1]{brion}. 
Additional material can be found
in~\cite{bezrukavnikov-braverman-positselskii, grosshans}. 
For general background on linear algebraic groups, see, e.g., 
\cite[Section~3.3]{dolgachev}~or~\cite{murnaghan}. 

%SF: I first wrote this, and then removed...
\iffalse
An algebraic group is \emph{solvable} if it has a composition series
of closed normal subgroups whose 
successive quotients are abelian subgroups.
Each algebraic group~$G$ contains a maximal connected solvable normal subgroup.
It is called the \emph{radical} of~$G$. 
The group~$G$ is called \emph{reductive} if its radical is a \emph{torus},
i.e., is isomorphic to (in the case of working over the complex numbers) $(\CC^*)^k$ for some~$k$. 
The group is \emph{semisimple} if its radical is trivial. 
\fi

\pagebreak[3]

Let $G$ be a simply connected semisimple complex algebraic group. 
Let~$U$ be a maximal unipotent subgroup of~$G$.
The variety $X=G/U$ is smooth and quasi-affine
(i.e., open in an affine variety). 
To be more specific, let~$\mathcal{O}(X)$ denote the ring of regular
functions on~$X$. 
Then $X$ embeds as an open subvariety into the affine (irreducible) 
variety $\overline{X}=\operatorname{Spec}(\mathcal{O}(X))$,
the ``affine completion'' of~$X$. 
The rings of regular functions on~$X$ and on~$\overline{X}$ 
coincide: $\mathcal{O}(X)=\mathcal{O}(\overline{X})$. 
Moreover these rings are naturally identified with the ring of
invariants~$\CC[G]^U$.
The variety $\overline{X}$ is called the \emph{basic affine space}
for~$G$.
It is normal and usually singular. 
%The codimension of $\overline{X}\setminus X$ in $\overline{X}$
%is~$\ge 2$. 

In this section we proved that $\CC[\SL_k]^U$ is a cluster algebra.
More generally, for a simply-laced $G$, with $U$ as above, 
there is a cluster algebra contained inside 
the coordinate ring $\CC[G]^U$ \cite{gls-partialflag}; it conjecturally 
coincides with the coordinate ring after a certain localization, see \cite[Conjecture~10.4]{gls-partialflag}.
In~this context, some \emph{generalized minors}
\cite[Definition~1.4]{fz-dbc} (see also 
\cite[Definition~6.2]{MR}) play the role of flag minors,
and are used to define a collection of special seeds for the corresponding cluster algebra.

The significance of the basic affine space stems from the well known
fact that its coordinate 
ring $\CC[G]^U$ is a direct sum of all irreducible rational
representations of~$G$, each occurring with multiplicity~$1$.
A more detailed description is as follows. 
The subgroup $U$ is the unipotent radical of a Borel 
subgroup~$B\!\subset \!G$. 
The action of the Cartan subgroup $H\!=\!B/U$~on~$X$
commutes with the natural $G$-action. 
The $H$-action induces a grading of~$\CC[G]^U$ 
by the weight lattice of~$G$. 
The graded components are labeled by the~dominant weights~$\lambda$, 
and carry irreducible representations of~$G$ (of highest
weight~$\lambda$).

When $G=\SL_k$ is the special linear group, 
the irreducible representation with the highest weight
$d_1 \omega_1 + d_2 \omega_2+\cdots$ 
(here $d_1, d_2, \ldots$ are nonnegative integers, 
and $\omega_1, \omega_2, \dots$ are the fundamental
weights of $G$ in the standard order) 
appears as the space of polynomials in the flag minors
which are homogeneous of degree $d_1$ in the flag minors 
$P_1, P_2, P_3, \dots$, of degree $d_2$ in $P_{12}, P_{13}, P_{23},\dots$,
and so~on.
To~rephrase, these polynomials have degree $d_1+d_2+d_3+\cdots$ with
respect to the first row entries of a $k\times k$ matrix;
degree $d_2+d_3+\cdots$ with
respect to the second row entries; and so on. 

A monomial in the flag minors is called a \emph{cluster
monomial} if all these minors belong to the same extended cluster. 
In the finite type cases where $G$ is $\SL_3$, $\SL_4$, or~$\SL_5$,
the cluster monomials form a $\CC$-basis of $\CC[G]^U$. 
This is an instance of (the classical limit of) 
the \emph{dual canonical basis} of G.~Lusztig~\cite{Lusztig-canonical}, 
also known as the  \emph{upper global basis} of M.~Kashiwara~\cite{Kashiwara}. 
(In the case $G=\SL_3$, this basis was introduced and studied in detail in~\cite{gz}.) 
This description of the basis served as the key original motivation 
for the introduction of cluster algebras in~\cite{ca1}.

In infinite type, the picture turns out to be much more complicated. 
The fundamental result obtained in~\cite{kkko} asserts that in general,
the dual canonical (or upper global) basis contains all cluster monomials. 
The rest of the basis still awaits an explicit description. 

\newpage

\section{Cluster structure in the rings
$\CC[\operatorname{Mat}_{k \times k}]$ and  $\CC[\SL_k]$}
\label{sec:rings-matrices}

The coordinate ring $\CC[\operatorname{Mat}_{k \times k}]$ of  
the space of $k \times k$ complex matrices is the polynomial ring~$\CC[z_{ij}]$;
here $z = (z_{ij})$ denotes a generic $k \times k$ matrix.
We will use double wiring diagrams to describe a cluster structure in this ring.
An adaptation of this construction will then produce a cluster structure in
the coordinate ring $\CC[\SL_k]$ of the special linear group.

\begin{theorem}\label{th:Matkk}
The ring 
$\CC[\operatorname{Mat}_{k \times k}]$
%$\CC[\SL_k]$ 
has a cluster structure whose set of coefficient and cluster variables
includes all minors of a $k\times k$ matrix. 
\end{theorem}

\begin{proof}
%Since $\CC[\SL_k]$ is generated by 
%the matrix entries of $z$, it is finitely generated.
Recall from Section \ref{sec:mut-double-wiring} and 
Exercise \ref{exercise:quiverexchange} that one can 
associate a seed $(\tilde \xx, \tilde B)$ to every double wiring diagram.  
All cluster and frozen variables in such a seed are minors of~$z$.
In particular, the set of frozen variables consists of all minors of the form 
$\Delta_{I,J}$ or $\Delta_{J,I}$ where 
$I = \{1,2,\dots, i\}$ and $J = \{k-i+1,k-i+2,\dots,k\}$, 
with $i\in\{1,\dots, k\}$.

Recall that any two double wiring diagrams can be 
connected by local moves, cf.\ Figure~\ref{fig:moves}. 
Since these moves correspond to mutations of the corresponding seeds,
all such seeds define the same cluster algebra. 

Let $(\tilde \xx, \tilde B)$ be a seed associated to a double wiring diagram.
To show that $\CC[\operatorname{Mat}_{k \times k}]
 \subset \AA(\tilde \xx, \tilde B)$, it suffices to show that each matrix entry~$z_{ij}$ lies in 
$\AA(\tilde \xx, \tilde B)$.  
The latter statement follows from the fact that one can always construct a double
wiring diagram whose extended cluster~$\tilde\xx$ contains~$z_{ij}$. 

The inclusion $\AA(\tilde \xx, \tilde B) \subset \CC[\operatorname{Mat}_{k \times k}]$ can be shown using either of the two arguments outlined below.
These two arguments use two different seeds as well as two 
	different versions of the Starfish lemma.
The first argument relies on
\cref{prop:cluster-criterion}, which requires checking a coprimality condition on cluster variables; 
	 the second argument uses
\cref{cor:cluster-criterion}, which requires checking an irreducibility condition.

The polynomial ring $\CC[z_{ij}]$ is 
%by Proposition~\ref{pr:voskresenskii}, the ring $\CC[\SL_k]$ is
factorial and therefore normal.  In order to use 
	\cref{prop:cluster-criterion}, we first observe that 
$\tilde \xx \subset \CC[\operatorname{Mat}_{k \times k}]$.
Moreover, any two cluster variables 
in $\tilde \xx$ are pairwise coprime, since the determinant is an 
irreducible polynomial.  It remains to check 
	property~(3) of \cref{prop:cluster-criterion}.
Let us choose the seed $(\tilde \xx, \tilde B)$ shown in Figure~\ref{fig:solid-minors}.
(The figure shows the example for $k=4$ but the generalization
%of this double wiring diagram and the corresponding seed 
to an arbitrary~$k$ is clear from the picture.)
Note that the minors appearing in this seed are very simple:
they are solid minors that ``stick" to the left edge or 
the upper edge of the matrix.
We now need to check that for each minor~$x_{\ell}$ in the seed,
the new cluster variable~$x'_{\ell}$ obtained by an exchange with~$x_\ell$ 
is a polynomial in the matrix entries that is moreover coprime to~$x_{\ell}$.
Although some cluster variables $x'_{\ell}$ will no longer
be minors (in particular, those resulting from mutations 
at degree~$6$ vertices in the quiver), 
one can adapt the argument from Section~\ref{sec:rings-baseaffine} to prove that they are nevertheless
polynomials coprime~to~$x_{\ell}$.

\pagebreak[3]

\begin{figure}[ht]
\begin{center}
\setlength{\unitlength}{1.22pt}
\begin{picture}(240,80)(6,0)
\thicklines
\dark{
\linethickness{1.8pt}
  \put(-10,0){\line(1,0){140}}
  \put(150,0){\line(1,0){30}}
  \put(200,0){\line(1,0){30}}
  \put(250,0){\line(1,0){10}}
  
  \put(-10,20){\line(1,0){140}}
  \put(150,20){\line(1,0){5}}
  \put(175,20){\line(1,0){5}}
  \put(200,20){\line(1,0){5}}
  \put(225,20){\line(1,0){5}}
  \put(250,20){\line(1,0){10}}

  \put(-10,40){\line(1,0){165}}
  \put(175,40){\line(1,0){5}}
  \put(200,40){\line(1,0){5}}
  \put(225,40){\line(1,0){35}}

  \put(-10,60){\line(1,0){190}}
  \put(200,60){\line(1,0){60}}
 % \put(60,0){\line(1,0){100}}
 % \put(180,0){\line(1,0){10}}
 % \put(0,20){\line(1,0){40}}
 % \put(60,20){\line(1,0){10}}
 % \put(90,20){\line(1,0){70}}
 % \put(180,20){\line(1,0){10}}
 % \put(0,40){\line(1,0){70}}
 % \put(90,40){\line(1,0){100}}
  
  \put(130,0){\line(1,1){20}}
  \put(180,0){\line(1,1){20}}
  \put(230,0){\line(1,1){20}}
  \put(130,20){\line(1,-1){20}}
  \put(180,20){\line(1,-1){20}}
  \put(230,20){\line(1,-1){20}}

  \put(155,20){\line(1,1){20}}
  \put(155,40){\line(1,-1){20}}
  \put(205,20){\line(1,1){20}}
  \put(205,40){\line(1,-1){20}}

  \put(180,40){\line(1,1){20}}
  \put(180,60){\line(1,-1){20}}

 % \put(40,0){\line(1,1){20}}
 % \put(70,20){\line(1,1){20}}
 % \put(160,0){\line(1,1){20}}

  \put(265,-6){$\mathbf{4}$}
  \put(265,14){$\mathbf{3}$}
  \put(265,34){$\mathbf{2}$}
  \put(265,54){$\mathbf{1}$}

  \put(-18,-6){$\mathbf{1}$}
  \put(-18,14){$\mathbf{2}$}
  \put(-18,34){$\mathbf{3}$}
  \put(-18,54){$\mathbf{4}$}
}

\light{
\thinlines

  \put(-10,2){\line(1,0){10}}
  \put(20,2){\line(1,0){30}}
  \put(70,2){\line(1,0){30}}
  \put(120,2){\line(1,0){140}}
 % \put(120,2){\line(1,0){70}}
  \put(-10,22){\line(1,0){10}}
  \put(20,22){\line(1,0){5}}
  \put(45,22){\line(1,0){5}}
  \put(70,22){\line(1,0){5}}
  \put(95,22){\line(1,0){5}}
  \put(120,22){\line(1,0){140}}
 % \put(120,22){\line(1,0){10}}
 % \put(150,22){\line(1,0){40}}
  \put(-10,42){\line(1,0){35}}
  \put(45,42){\line(1,0){5}}
  \put(70,42){\line(1,0){5}}
  \put(95,42){\line(1,0){165}}
   \put(-10,62){\line(1,0){60}}
   \put(70,62){\line(1,0){190}}

  \put(0,2){\line(1,1){20}}
  \put(50,2){\line(1,1){20}}
  \put(100,2){\line(1,1){20}}

  \put(25,22){\line(1,1){20}}
  \put(75,22){\line(1,1){20}}

  \put(50,42){\line(1,1){20}}

  %\put(100,2){\line(1,1){20}}
  %\put(130,22){\line(1,1){20}}

  \put(0,22){\line(1,-1){20}}
  \put(50,22){\line(1,-1){20}}
  \put(100,22){\line(1,-1){20}}

  \put(25,42){\line(1,-1){20}}
  \put(75,42){\line(1,-1){20}}

  \put(50,62){\line(1,-1){20}}

  %\put(100,22){\line(1,-1){20}}
  %\put(130,42){\line(1,-1){20}}

  \put(265,2){${1}$}
  \put(265,22){${2}$}
  \put(265,42){${3}$}
  \put(265,62){${4}$}

  \put(-18,2){${4}$}
  \put(-18,22){${3}$}
  \put(-18,42){${2}$}
  \put(-18,62){${1}$}

%  \put(94,-10){$_{\light{\emptyset},\dark{\emptyset\hspace{-.057in}
%\emptyset\hspace{-.057in}\emptyset}}$}

  \put(-4,10){$_{\light{4},\mathbf{\dark{1}}}$}
  \put(30,10){$_{\light{3},\mathbf{\dark{1}}}$}
  \put(80,10){$_{\light{2},\mathbf{\dark{1}}}$}
  \put(120,10){$_{\light{1},\mathbf{\dark{1}}}$}
  \put(160,10){$_{\light{1},\mathbf{\dark{2}}}$}
  \put(210,10){$_{\light{1},\mathbf{\dark{3}}}$}
  \put(250,10){$_{\light{1},\mathbf{\dark{4}}}$}
  %\put(76,10){$_{\light{3},\mathbf{\dark{2}}}$}
  %\put(136,10){$_{\light{1},\mathbf{\dark{2}}}$}
  %\put(181,10){$_{\light{1},\mathbf{\dark{3}}}$}

  \put(5,30){$_{\light{34},\mathbf{\dark{12}}}$}
  \put(55,30){$_{\light{23},\mathbf{\dark{12}}}$}
  \put(115,30){$_{\light{12},\mathbf{\dark{12}}}$}
  \put(180,30){$_{\light{12},\mathbf{\dark{23}}}$}
  \put(230,30){$_{\light{12},\mathbf{\dark{34}}}$}
  %\put(42,30){$_{\light{13},\mathbf{\dark{12}}}$}
  %\put(102,30){$_{\light{13},\mathbf{\dark{23}}}$}
  %\put(162,30){$_{\light{12},\mathbf{\dark{23}}}$}

  \put(25,50){$_{\light{234},\mathbf{\dark{123}}}$}
  \put(115,50){$_{\light{123},\mathbf{\dark{123}}}$}
  \put(205,50){$_{\light{123},\mathbf{\dark{234}}}$}

  \put(115,70){$_{\light{1234},\mathbf{\dark{1234}}}$}

}
\end{picture}
\end{center}
\caption{A double wiring diagram~$D$ whose extended cluster consists of solid minors.
In the corresponding quiver, 
many mutable vertices 
%every mutable vertex coming from
%a bounded chamber in the ``bulk'' of one of two ``halves'' 
will have degree~$6$. 
Mutation at such a vertex will result in a cluster variable which is not a minor.}
\label{fig:solid-minors}
\end{figure}
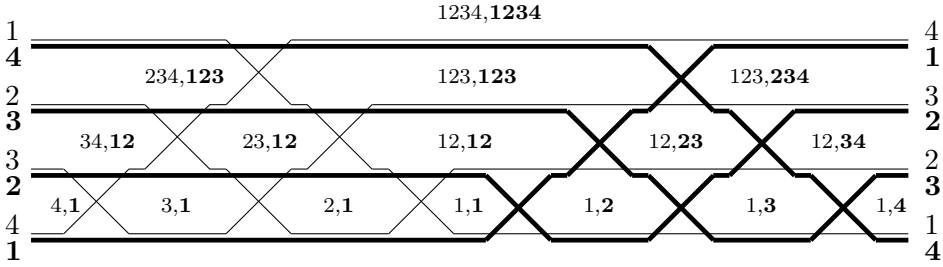

\medskip

An alternative approach relies on \cref{cor:cluster-criterion}
to establish the inclusion $\AA(\tilde \xx, \tilde B) \subset \CC[\operatorname{Mat}_{k \times k}]$.
Here we use a 
``grid seed" for $\CC[\operatorname{Mat}_{k \times k}]$, see 
Figures~\ref{fig:better} and~\ref{fig:gridquiver}.
%Figure \ref{fig:grid-SL}.
%To describe the minors that appear in the cluster of such a 
%grid seed, let $j = \lfloor \frac{k}{2} \rfloor.$  Then 
%\Comment{Change the seed and figure; \cref{fig:better} from Jan 20 --
%We probably don't need to list all the minors, but give
%wiring diagram and the quiver.}
%the $1 \times 1$ minors in the cluster are:
%\begin{itemize}[leftmargin=.2in]
%\item $\Delta_{i,1}$ for $i = j+1, j+2, \dots, k-1$
%\item $\Delta_{i,k}$ for $i = 2, 3, \dots, j+1$.
%\end{itemize}
%The $2 \times 2$ minors in the cluster are:
%\begin{itemize}[leftmargin=.2in]
%\item $\Delta_{\{i,i+1\},\{1,2\}}$ for $i=j+1, j+2,\dots, k-2$
%\item $\Delta_{\{i,i+1\},\{1,k\}}$ for $i=j, j+1$
%\item $\Delta_{\{i,i+1\},\{k-1,k\}}$ for $i=2,3,\dots, j$.
%\end{itemize}
%The $3 \times 3$ minors in the cluster are:
%\begin{itemize}[leftmargin=.2in]
%\item $\Delta_{\{i,i+1,i+2\},\{1,2,3\}}$ for $i=j+1,\dots, k-3$
%\item $\Delta_{\{i,i+1,i+2\},\{1,2,k\}}$ for $i=j, j+1$
%\item $\Delta_{\{i, i+1,i+2\},\{1,k-1,k\}}$ for $i=j-1,j$
%\item $\Delta_{\{i,i+1,i+2\},\{k-2,k-1,k\}}$ for $i=2,3,\dots, j-1$.
%\end{itemize}
%And so on.
%
This seed has the property that every mutable vertex in its quiver
has degree three or four, and the corresponding exchange relation is 
a three-term Grassmann-Pl\"ucker relation.  It follows that for each 
cluster variable $x_{\ell}$ the adjacent cluster variable $x'_{\ell}$
is a minor, and hence
an irreducible polynomial.  
The claim $\AA(\tilde \xx, \tilde B) \subset \CC[\operatorname{Mat}_{k \times k}]$ follows.
%Therefore $\CC[\operatorname{Mat}_{k \times k}]=\AA(\tilde \xx, \tilde B)$, 
%which completes the proof.
\end{proof}

%\begin{figure}
%	\includegraphics[height=2in]{Fig6point5}
%	\caption{A double wiring diagram whose corresponding quiver is 
%	the ``grid quiver," see \cref{fig:gridquiver}, in which every 
%	mutable vertex in the quiver has degree three or four. 
%	\Comment{TO DO: Use the code for \cref{fig:solid-minors} as a model for how to 
%	make this figure.  (But don't delete \cref{fig:solid-minors}, we are keeping both \cref{fig:solid-minors} and this one.)}}
%	\label{fig:better}
%\end{figure}
%\Comment{Have old photo 0118.eps of handdrawn figure from 
%visit with Sergey; but its quiver needs to have arrows reversed.}
\begin{figure}[ht]
\begin{center}
\setlength{\unitlength}{1.22pt}
\begin{picture}(240,80)(6,0)
\thicklines
\dark{
\linethickness{1.8pt}
  \put(-10,0){\line(1,0){95}}
  \put(105,0){\line(1,0){100}}
  \put(225,0){\line(1,0){35}}
  
  \put(-10,20){\line(1,0){35}}
  \put(45,20){\line(1,0){40}}
  \put(105,20){\line(1,0){40}}
  \put(165,20){\line(1,0){40}}
  \put(225,20){\line(1,0){35}}

  \put(-10,40){\line(1,0){35}}
  \put(45,40){\line(1,0){40}}
  \put(105,40){\line(1,0){40}}
  \put(165,40){\line(1,0){40}}
  \put(225,40){\line(1,0){35}}

  \put(-10,60){\line(1,0){95}}
  \put(105,60){\line(1,0){100}}
  \put(225,60){\line(1,0){35}}
  
  \put(85,0){\line(1,1){20}}
  \put(205,0){\line(1,1){20}}
  \put(85,20){\line(1,-1){20}}
  \put(205,20){\line(1,-1){20}}

  \put(25,20){\line(1,1){20}}
  \put(25,40){\line(1,-1){20}}
    \put(145,20){\line(1,1){20}}
  \put(145,40){\line(1,-1){20}}

  \put(85,40){\line(1,1){20}}
  \put(85,60){\line(1,-1){20}}
  \put(205,40){\line(1,1){20}}
  \put(205,60){\line(1,-1){20}}

 % \put(40,0){\line(1,1){20}}
 % \put(70,20){\line(1,1){20}}
 % \put(160,0){\line(1,1){20}}

  \put(265,-6){$\mathbf{4}$}
  \put(265,14){$\mathbf{3}$}
  \put(265,34){$\mathbf{2}$}
  \put(265,54){$\mathbf{1}$}

  \put(-18,-6){$\mathbf{1}$}
  \put(-18,14){$\mathbf{2}$}
  \put(-18,34){$\mathbf{3}$}
  \put(-18,54){$\mathbf{4}$}
}

\light{
\thinlines

  \put(-10,2){\line(1,0){35}}
  \put(45,2){\line(1,0){100}}
  \put(165,2){\line(1,0){95}}

  \put(-10,22){\line(1,0){35}}
  \put(45,22){\line(1,0){40}}
  \put(105,22){\line(1,0){40}}
  \put(165,22){\line(1,0){40}}
  \put(225,22){\line(1,0){35}}
    \put(-10,42){\line(1,0){35}}
  \put(45,42){\line(1,0){40}}
  \put(105,42){\line(1,0){40}}
  \put(165,42){\line(1,0){40}}
  \put(225,42){\line(1,0){35}}
  
  \put(-10,62){\line(1,0){35}}
  \put(45,62){\line(1,0){100}}
  \put(165,62){\line(1,0){95}}

  \put(25,2){\line(1,1){20}}
  \put(145,2){\line(1,1){20}}

  \put(85,22){\line(1,1){20}}
  \put(205,22){\line(1,1){20}}

  \put(25,42){\line(1,1){20}}
  \put(145,42){\line(1,1){20}}

  \put(25,22){\line(1,-1){20}}
  \put(145,22){\line(1,-1){20}}

  \put(85,42){\line(1,-1){20}}
  \put(205,42){\line(1,-1){20}}

  \put(25,62){\line(1,-1){20}}
  \put(145,62){\line(1,-1){20}}

  \put(265,2){${1}$}
  \put(265,22){${2}$}
  \put(265,42){${3}$}
  \put(265,62){${4}$}

  \put(-18,2){${4}$}
  \put(-18,22){${3}$}
  \put(-18,42){${2}$}
  \put(-18,62){${1}$}

  \put(-5,10){$_{\light{4},\mathbf{\dark{1}}}$}
  \put(50,10){$_{\light{3},\mathbf{\dark{1}}}$}
  \put(115,10){$_{\light{3},\mathbf{\dark{3}}}$}
  \put(175,10){$_{\light{1},\mathbf{\dark{3}}}$}
  \put(230,10){$_{\light{1},\mathbf{\dark{4}}}$}
  
  \put(-5,30){$_{\light{34},\mathbf{\dark{12}}}$}
  \put(50,30){$_{\light{34},\mathbf{\dark{13}}}$}
  \put(115,30){$_{\light{13},\mathbf{\dark{13}}}$}
  \put(175,30){$_{\light{13},\mathbf{\dark{34}}}$}
  \put(230,30){$_{\light{12},\mathbf{\dark{34}}}$}

  \put(-5,50){$_{\light{234},\mathbf{\dark{123}}}$}
  \put(50,50){$_{\light{134},\mathbf{\dark{123}}}$}
  \put(115,50){$_{\light{134},\mathbf{\dark{134}}}$}
  \put(175,50){$_{\light{123},\mathbf{\dark{134}}}$}
  \put(230,50){$_{\light{123},\mathbf{\dark{234}}}$}

  \put(115,70){$_{\light{1234},\mathbf{\dark{1234}}}$}

}
\end{picture}
\end{center}
\caption{A double wiring diagram for $\CC[\operatorname{Mat}_{4 \times 4}]$
whose associated quiver is the ``grid quiver" shown in \cref{fig:gridquiver}.} 
\label{fig:better}
\end{figure}

%\begin{figure}
%	\includegraphics[height=2in]{Fig6point6}
%	\caption{\Comment{TO DO: Use the code for \cref{fig:grid-SL} as a model for how to 
%	make this figure.  Then we will delete \cref{fig:grid-SL}.}}
%	\label{fig:gridquiver}
%\end{figure}

\begin{figure}[ht]
\begin{center}
\setlength{\unitlength}{2pt} 
\begin{picture}(160,60)(0,0) 

%Boxes grouped by row
\put(80,60){\makebox(0,0){$\boxed{1234,\mathbf{1234}}$}} 

\put(-1,40){\makebox(0,0){$\boxed{234,\mathbf{123}}$}} 
\put(40,40){\makebox(0,0){${134,\mathbf{123}}$}} 
\put(80,40){\makebox(0,0){${134,\mathbf{134}}$}} 
\put(120,40){\makebox(0,0){${123,\mathbf{134}}$}} 
\put(160,40){\makebox(0,0){$\boxed{123,\mathbf{234}}$}} 

\put(-1,20){\makebox(0,0){$\boxed{34,\mathbf{12}}$}}
\put(40,20){\makebox(0,0){${34,\mathbf{13}}$}} 
\put(80,20){\makebox(0,0){${13,\mathbf{13}}$}} 
\put(120,20){\makebox(0,0){${13,\mathbf{34}}$}} 
\put(160,20){\makebox(0,0){$\boxed{12,\mathbf{34}}$}} 

\put(-1,0){\makebox(0,0){$\boxed{4,\mathbf{1}}$}} 
\put(40,0){\makebox(0,0){${3,\mathbf{1}}$}} 
\put(80,0){\makebox(0,0){${3,\mathbf{3}}$}} 
\put(120,0){\makebox(0,0){${1,\mathbf{3}}$}} 
\put(160,0){\makebox(0,0){$\boxed{1,\mathbf{4}}$}}

%Diagonal vectors
\put(72,54){\vector(-2,-1){22}}
\put(88,54){\vector(2,-1){22}}

%Horizontal vectors grouped by row
\put(29,40){\vector(-1,0){17}}
\put(52,40){\vector(1,0){17}}
\put(108,40){\vector(-1,0){17}}
\put(132,40){\vector(1,0){15}}

\put(10,20){\vector(1,0){21}}
\put(71,20){\vector(-1,0){22}}
\put(89,20){\vector(1,0){22}}
\put(149,20){\vector(-1,0){20}}

\put(33,0){\vector(-1,0){26}}
\put(47,0){\vector(1,0){27}}
\put(113,0){\vector(-1,0){26}}
\put(127,0){\vector(1,0){25}}

%Vertical vectors grouped by column
\put(40,24){\vector(0,1){12}}
\put(40,16){\vector(0,-1){12}}

\put(80,44){\vector(0,1){11}}
\put(80,36){\vector(0,-1){12}}
\put(80,4){\vector(0,1){12}}

\put(120,24){\vector(0,1){12}}
\put(120,16){\vector(0,-1){12}}

\end{picture} 
\end{center}
\caption{A grid seed for $\CC[\operatorname{Mat}_{4 \times 4}]$, cf.\ \cref{fig:better}.
Every mutable vertex in the quiver has degree three or four.}
\label{fig:gridquiver}
\end{figure}

\begin{theorem}\label{th:SLk}
The coordinate ring 
%$\CC[\operatorname{Mat_{k \times k}}]$
$\CC[\SL_k]$ of the special linear group has a cluster structure 
whose set of coefficient and cluster variables
includes all minors of a $k\times k$ matrix, except for the determinant of the matrix. 
\end{theorem}

\begin{proof}
To adapt the above arguments  to the case of $\CC[\SL_k]$, we use 
Proposition~\ref{pr:voskresenskii} to show that the ring $\CC[\SL_k]$ is
factorial, and hence  normal.  
The cluster variables, frozen variables, and clusters are the same as for 
$\CC[\operatorname{Mat}_{k \times k}]$, except that the 
$k \times k$ determinant of the entire matrix is no longer a frozen variable 
(as it is now equal to~1).
\end{proof}

\begin{remark}
One might expect that the cluster structures on 
$\CC[\operatorname{Mat}_{k \times k}]$ and $\CC[\SL_k]$ described in this section 
can be modified to yield a cluster structure in the coordinate ring
of a general linear group~$\GL_k$.  
%However, $\CC[\GL_k]$ is not a cluster algebra in our sense,
However this cannot be achieved without tweaking the basic definitions,  
because the inverse of the determinant $\det^{-1}\in\CC[\GL_k]$ 
%inverse $\Delta^{-1}$ of the $k \times k$ determinant is 
is a regular function that does not lie in the cluster algebra. 
(The ground ring for a cluster algebra is the polynomial ring
generated by the coefficient variables; it does not include their inverses.) 
As noted in Definition~\ref{def:cluster-algebra}, a common alternative
is to change the ground ring, adjoining the inverses of the coefficient variables 
(or ``localizing at coefficients''). 
With this convention, the coordinate ring $\CC[\GL_k]$ becomes a cluster algebra.
\end{remark}

\begin{remark}
The constructions presented above allow multiple generalizations and variations.  
In particular, one can replace $\SL_k$ by any connected, 
simply connected semisimple complex Lie group~$G$
and/or consider various subvarieties of~$G$, such as those related to \emph{double Bruhat cells}, 
see~\cite{ca3}.
\end{remark}

%\Comment{Need to talk about cluster algebras where we 
%localize at coefficients.  Examples such as $\CC[\GL_k]$
%and the double Bruhat cells.  We have an issue whenever
%our rings have non-trivial units.
%Whenever we write this section, we should refer back to 
%the parenthetical comment of Definition \ref{def:cluster-algebra};
%that definition now includes a forward reference
%to this section.}

%\newpage

\section{The cluster structure in the ring $\CGr{a}{b}$}
\label{sec:plucker-rings}

%$\CGr{a}{b}$

%\Comment{Want to emphasize that cluster structure 
%of $Gr(a,b-1)$ is subpattern of that for $Gr(a,b)$ -- by cutting off
%rightmost column of quive.r  Also, by cutting off bottom-most
%row of quiver and deleting the index $1$, we get that 
%the pattern for $Gr(a-1,b-1)$ is a subpattern of that for $Gr(a,b)$.}

%\Comment{Move grid quiver to periodicity chapter, and that chapter should
%come after Grassmannians chapter and add exercise that
% grid quiver mutation equivalent to Le-diagram quiver.}

The Grassmannian  $\operatorname{Gr}_{a,b}$ of $a$-dimensional subspaces
in~$\CC^b$ can be embedded in the projective space of dimension $\binom{b}{a}-1$ 
via the Pl\"ucker embedding; see, e.g., \cite[Corollary~2.3]{dolgachev}. 
Let $\widehat{\Gr}_{a,b}$ denote the affine cone over ${\rm Gr}_{a,b}$ 
taken in this embedding. 
The ring $\CGr{a}{b}$  
(the homogeneous coordinate ring of~$\operatorname{Gr}_{a,b}$) 
is generated by the Pl\"ucker coordinates~$P_J$,
where $J$ ranges over all $a$-element subsets of $\{1,\dots,b\}$).
These generators satisfy the quadratic \emph{Grassmann-Pl\"ucker relations}.

\begin{example}[cf.\ Section~\ref{sec:Ptolemy}]
The homogeneous
 coordinate ring $\CGr{2}{4}$ is generated by the six Pl\"ucker coordinates
$P_{12},P_{13},P_{14}, P_{23}, P_{24}, P_{34}$, 
which are subject to the single Grassmann-Pl\"ucker relation
\begin{equation}
\label{eq:13*24=}
P_{13} P_{24} = P_{12} P_{34} + P_{14} P_{23}\,. 
\end{equation}
This ring carries the structure of a cluster algebra of rank~1 in which
\begin{itemize}[leftmargin=.2in]
\item
the ambient field is the field $\CC(P_{12},P_{13},P_{14}, P_{23}, P_{34})$ 
of rational functions in five algebraically independent variables; 
\item 
the frozen variables are $P_{12},P_{23},P_{34},P_{14}$; 
\item
the cluster variables are $P_{13}$ and $P_{24}$; 
\item
the single exchange relation is \eqref{eq:13*24=}. 
\end{itemize}
\iffalse
The correspondence 
$$x \mapsto P_{13}, \quad \overline x \mapsto P_{24}, 
\quad p^+ \mapsto P_{12} P_{34}, \quad p^- \mapsto P_{14} P_{23}$$ 
identifies $\AA$ with the subring of $\CC[X]$ generated by 
$P_{13}, P_{24}, P_{12} P_{34}$, and $P_{14} P_{23}$. 
%It is easy to see that 
This ring is a $\ZZ$-form of the ring of invariants 
$\CC[X]^T$, where 
$T \subset SL_4$ is the torus of all 
diagonal matrices of the form 
$$\left[\!\!\begin{array}{cccc} 
t_1 & 0 & 0 & 0\\ 
0 & t_2 & 0 & 0\\ 
0 & 0 & t_1^{-1} & 0\\ 
0 & 0 & 0 & t_2^{-1}\\ 
\end{array}\!\!\right] \, ,$$ 
naturally acting on $X$. 
\fi
\end{example}

The ring $\CGr{a}{b}$ is an archetypal object of classical 
invariant theory; see, e.g., \cite[Chapter~2]{dolgachev}, 
\cite[\S9]{popov-vinberg},
\cite[Section~11]{procesi-textbook},
and~\cite{weyl}. %\linebreak[3]
In invariant theory, this ring is typically given a somewhat different
description: 

\begin{definition}
Let $V\cong\CC^a$ be an $a$-dimensional complex vector space equipped with a volume form. 
The special linear group $\SL(V)\cong \SL_a(\CC)$ 
naturally acts on the vector space~$V^b$
of $b$-tuples of vectors, 
hence on its coordinate (polynomial) ring. \linebreak[3]
The \emph{Pl\"ucker ring}~$R_{a,b}$ is the ring 
\begin{equation}
\label{eq:Rkb}
R_{a,b}=\CC[V^b]^{\SL(V)}
\end{equation}
of $\SL(V)$-invariant polynomials on~$V^b$. 
As a subring of a polynomial ring, $R_{a,b}$ is a domain. 

In coordinate notation, the Pl\"ucker ring is described as follows.
Consider a matrix $z=(z_{ij})$ of size $a\times b$
filled with indeterminates.
The ring $R_{a,b}$ consists of polynomials in these $ab$
variables that are invariant under the transformations $z\mapsto gz$,
for $g\in\SL_a(\CC)$. 
One example of such a polynomial is a \emph{Pl\"ucker coordinate}~$P_J$
where $J$ is an $a$-element subset of columns in~$z$;
by definition, $P_J$ is the $a\times a$ minor of~$z$ 
occupying the columns in~$J$. 
\end{definition}

The First Fundamental Theorem of invariant theory,
which goes back to A.~Clebsch and H.~Weyl, 
states the following. 

\begin{theorem}
\label{th:FFT}
The Pl\"ucker ring $R_{a,b}$ is generated by the 
$\binom{b}{a}$ Pl\"ucker coordinates~$P_J$. 
\end{theorem}

Theorem~\ref{th:FFT} implies that for $a \leq b$, 
the Pl\"ucker ring $R_{a,b}$ is isomorphic to~$\CGr{a}{b}$, 
the homogeneous coordinate ring  of the 
Grassmannian~$\Gr_{a,b}\,$.  Therefore we can talk interchangeably about these
two rings.

We note that the fact that the Pl\"ucker ring is finitely generated
%(cf.\ Theorem~\ref{th:FFT}) 
is a special case of Theorem~\ref{th:hilbert}.

\begin{remark}
The Second Fundamental Theorem, which we will not need, 
describes the ideal of relations among the generators~$P_J$ of the Pl\"ucker ring~$R_{a,b}$. 
As mentioned above, this ideal is generated by certain quadratic relations, 
the Grassmann-Pl\"ucker relations. 
The 3-term Grassmann-Pl\"ucker relations are among the exchange relations
of the standard cluster structure on~$R_{a,b}$ described below.
When $3\le a\le b-3$, the Grassmann-Pl\"ucker relations include some 
longer quadratic relations which are not generated by the 3-term ones, 
cf.~Example~\ref{ex:Gr(3,6)} below. 
\end{remark}

%(This interpretation of the affine cone over the Grassmannian as a
%GIT quotient is sometimes called \emph{Gelfand-MacPherson
%  correspondence}.) 

%SF: unclear if we need this
\iffalse
There are known combinatorial formulas 
for the dimensions of multigraded components of the ring~$R_{k,b}$
in terms of the iterated 
Littlewood-Richardson Rule. 
In some cases, more explicit formulas can be given. 
For example, the multilinear component of $R_{k,b}$
(i.e., the space of $\SL_k$-invariant $b$-covariant 
tensors in~$\CC^k$) is only nontrivial when $k$ divides~$b$;
if it does, then the dimension of this space is equal to the number of standard Young
tableaux of rectangular shape $k\times \frac{b}{k}$. 
As such, this dimension is given by the corresponding instance of the hooklength formula. 
\fi

\begin{corollary}[{\cite[Section~1.6b]{popov-vinberg-quasihomogeneous}}]
The Pl\"ucker ring $R_{a,b}$ is
factorial.
\end{corollary}

\begin{proof}
This follows from Theorem~\ref{th:inv-ring-is-factorial} and~\eqref{eq:Rkb}. 
(Being semisimple, $\SL_a$~has no nontrivial characters.)
\end{proof}

We next describe a cluster structure in the Pl\"ucker ring~$R_{a,b}$
\cite{scott}.
While canonical up to a ring automorphism, 
this structure will depend on the choice of a \emph{cyclic ordering} of the $b$
vectors. 

The set of coefficient variables for this cluster structure in~$R_{a,b}$
consists of the $b$ Pl\"ucker coordinates~$P_J$ where $J$ is a contiguous segment modulo~$b$.   
%We call such Pl\"ucker coordinates
%\emph{solid}.
For example, the coefficient variables for $R_{3,7}$ are the Pl\"ucker coordinates
\[
P_{123}, P_{234}, P_{345}, P_{456}, P_{567}, P_{167}, P_{127}. 
\]

We will work with some distinguished seeds: 
the \emph{rectangles seed} $\Sigma_{a,b}$, together with its cyclic shifts
$\Sigma_{a,b}^i$ for $1 \leq i \leq b-1$.
To define the rectangles seed $\Sigma_{a,b}$, we first
construct a quiver $Q_{a,b}$ whose vertices are labeled by the 
rectangles contained in an $a \times (b-a)$ rectangle~$R$,
including the empty rectangle~$\varnothing$.
The frozen vertices of $Q_{a,b}$ are labeled by 
the rectangles of sizes $a\times j$ (with $1\le j\le b-a$), 
rectangles of sizes $i\times (b-a)$ (with $1\le i\le a$), and the empty rectangle. 
The arrows from an $i\times j$ rectangle go to rectangles of sizes $i\times (j+1)$, 
$(i+1)\times j$, and $(i-1)\times (j-1)$ (assuming those rectangles have nonzero 
dimensions, fit inside~$R$, and the arrow does not connect two frozen vertices).
There is also an arrow from the frozen vertex labeled by~$\varnothing$
to the vertex labeled by the $1\times 1$ rectangle. 
See Figure~\ref{fig:G37-Le-quiver2}.  

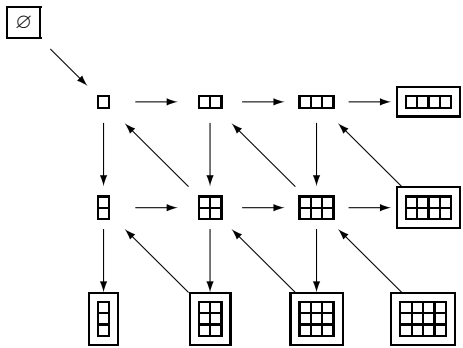
\begin{figure}[ht]
\begin{center}
\setlength{\unitlength}{2pt} 
\begin{picture}(80,55)(0,-1) 
\put(20,20){\makebox(0,0){${\ydiagram{1,1}}$}} 
\put(20,40){\makebox(0,0){${\ydiagram{1}}$}} 
\put(40,20){\makebox(0,0){${\ydiagram{2,2}}$}} 
\put(40,40){\makebox(0,0){${\ydiagram{2}}$}} 
\put(60,20){\makebox(0,0){${\ydiagram{3,3}}$}} 
\put(60,40){\makebox(0,0){${\ydiagram{3}}$}} 

%\put(-1,61){\makebox(0,0){$\boxed{\emptyset}$}} 
\put(5,55){\makebox(0,0){$\boxed{\scriptstyle\varnothing}$}} 
\put(20,-1){\makebox(0,0){$\boxed{\ydiagram{1,1,1}}$}} 
\put(40,-1){\makebox(0,0){$\boxed{\ydiagram{2,2,2}}$}} 
\put(60,-1){\makebox(0,0){$\boxed{\ydiagram{3,3,3}}$}} 
\put(80,-1){\makebox(0,0){$\boxed{\ydiagram{4,4,4}}$}} 
\put(81,20){\makebox(0,0){$\boxed{\ydiagram{4,4}}$}} 
\put(81,40){\makebox(0,0){$\boxed{\ydiagram{4}}$}} 

\put(20,16){\vector(0,-1){12}}
\put(40,36){\vector(0,-1){12}}

\put(20,36){\vector(0,-1){12}}
\put(60,36){\vector(0,-1){12}}

\put(26,20){\vector(1,0){8}}
\put(66,20){\vector(1,0){8}}
\put(46,40){\vector(1,0){8}}

\put(26,40){\vector(1,0){8}}
\put(66,40){\vector(1,0){8}}
\put(46,20){\vector(1,0){8}}

%\put(4,56){\vector(1,-1){12}}
\put(10,50){\vector(1,-1){7}}

\put(40,16){\vector(0,-1){12}}
\put(60,16){\vector(0,-1){12}}

\put(36, 4){\vector(-1,1){12}}
\put(56, 4){\vector(-1,1){12}}
\put(76, 4){\vector(-1,1){12}}
\put(36, 24){\vector(-1,1){12}}
\put(56, 24){\vector(-1,1){12}}
\put(76, 24){\vector(-1,1){12}}
\end{picture} 
\end{center}
\caption{The quiver $Q_{3,7}$. % for the rectangles seed $\Sigma_{3,7}$. 
The vertices are labeled by rectangles contained in a $3 \times 4$ rectangle,
and arranged in a (triangulated) grid. 
The width and height of the rectangles increase
from left to right and from top to bottom, respectively.  
%Diagonal arrows provide a triangulation of the grid, as shown in the figure.
}
\label{fig:G37-Le-quiver2}
\end{figure}

We map each rectangle $r$ contained in the $a \times (b-a)$ rectangle
$R$ to an $a$-element subset of 
$\{1,2,\dots, b\}$ (representing a Pl\"ucker coordinate), as follows. 
We justify $r$  
so that its upper left corner coincides with the upper left corner
of $R$.  There is a path of length $b$ from the northeast corner of 
$R$ to the southwest corner of $R$ which cuts out the 
smaller rectangle $r$;  we label the steps of this path from $1$ to~$b$. 
We then map $r$ to the set of labels $J(r)$ of the vertical steps on this path. 
This construction allows us to assign to each vertex of the quiver $Q_{a,b}$ 
a particular Pl\"ucker coordinate. 
We set 
%\begin{definition}
\begin{equation*}
\tilde\xx^{a,b} = 
 \{P_{J(r)} \ \vert \ \text{ $r$ is a rectangle contained in an
$a \times (b-a)$ rectangle}\} , 
\end{equation*}
and then define the \emph{rectangles seed} $\Sigma_{a,b}=(\tilde\xx^{a,b}, \tilde B(Q_{a,b}))$.
See Figure~\ref{fig:G37-Le-quiver}.

\begin{figure}[ht]
\begin{center}
\setlength{\unitlength}{2pt} 
\begin{picture}(80,65)(0,0) 
\put(20,20){\makebox(0,0){${457}$}} 
\put(20,40){\makebox(0,0){${467}$}} 
\put(40,20){\makebox(0,0){${347}$}} 
\put(40,40){\makebox(0,0){${367}$}} 
\put(60,20){\makebox(0,0){${237}$}} 
\put(60,40){\makebox(0,0){${267}$}} 

\put(-1,61){\makebox(0,0){$\boxed{567}$}} 
\put(20,-1){\makebox(0,0){$\boxed{456}$}} 
\put(40,-1){\makebox(0,0){$\boxed{345}$}} 
\put(60,-1){\makebox(0,0){$\boxed{234}$}} 
\put(80,-1){\makebox(0,0){$\boxed{123}$}} 
\put(81,20){\makebox(0,0){$\boxed{127}$}} 
\put(81,40){\makebox(0,0){$\boxed{167}$}} 

\put(20,16){\vector(0,-1){12}}
\put(40,36){\vector(0,-1){12}}

\put(20,36){\vector(0,-1){12}}
\put(60,36){\vector(0,-1){12}}

\put(26,20){\vector(1,0){8}}
\put(66,20){\vector(1,0){8}}
\put(46,40){\vector(1,0){8}}

\put(26,40){\vector(1,0){8}}
\put(66,40){\vector(1,0){8}}
\put(46,20){\vector(1,0){8}}

\put(4,56){\vector(1,-1){12}}

\put(40,16){\vector(0,-1){12}}
\put(60,16){\vector(0,-1){12}}

\put(36, 4){\vector(-1,1){12}}
\put(56, 4){\vector(-1,1){12}}
\put(76, 4){\vector(-1,1){12}}
\put(36, 24){\vector(-1,1){12}}
\put(56, 24){\vector(-1,1){12}}
\put(76, 24){\vector(-1,1){12}}
\end{picture} 
\end{center}
\caption{The rectangles seed $\Sigma_{3,7}$.}
\label{fig:G37-Le-quiver}
\end{figure}
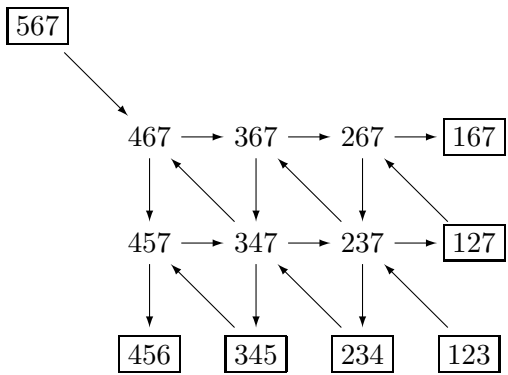

\begin{remark}
When $a=2$, the rectangles seed $\Sigma_{2,b}$ coincides with the 
seed associated to the triangulation of the polygon $\mathbf{P}_b$
that uses all of the diagonals incident to the vertex~$b$,
cf.\ \cref{ex:realization-An} and Definition~\ref{def:Q(T)-polygon}.
\end{remark}

Given an $a$-element subset $J=\{j_1,j_2, \dots, j_a\}\subset\{1,2,\dots, b\}$ 
and a positive integer $i$, 
we define 
\[
(J+i)\bmod b = \{j_1+i, j_2+i, \dots, j_a+i\}, 
\]
where the sums are taken modulo~$b$.  
We define a quiver $Q_{a,b}^i$ and seed $\Sigma_{a,b}^i$ 
%from $Q_{a,b}$ and $\Sigma_{a,b}$ 
by replacing each vertex 
label
$J$ in $Q_{a,b}$ by $(J+i)\bmod b$.  (The quivers $Q_{a,b}$ and $Q_{a,b}^i$ are 
exactly the same; only their vertex labels are different.)

\begin{exercise}\label{ex:cyclic}
Start from the seed $\Sigma_{a,b}$ and mutate at each of the mutable vertices 
of $Q_{a,b}$ exactly once, in the following order:
mutate each row from left to right, 
starting from the bottom row and ending at the top row. 
Show that at the end of this process, one obtains the seed~$\Sigma_{a,b}^1$.  
For example, in Figure~\ref{fig:G37-Le-quiver}, 
mutating at $457$, $347$, $237$, $467$, $367$, $267$ (in this order) 
recovers the same quiver but with each label cyclically shifted.

The same mutation sequence transforms the seed $\Sigma_{a,b}^i$ into 
$\Sigma_{a,b}^{i+1}$.
Therefore the rectangles seed and all its cyclic shifts are 
mutation equivalent.
\end{exercise}

\begin{theorem}
\label{th:Grassmain}
The seed pattern defined by $\Sigma_{a,b}$ 
(or by any of its cyclic shifts, cf.\ Exercise \ref{ex:cyclic}) gives
the Pl\"ucker ring $R_{a,b} = \CGr{a}{b}$ 
the structure of a cluster algebra.  
\end{theorem}

Theorem~\ref{th:Grassmain} was first proved in \cite{scott}, using results from~\cite{postnikov}.  
Below in this section we give a different (and self-contained) proof.

\begin{remark}
There is a well known isomorphism $R_{a,b}\to R_{b-a,b}$
defined by $P_J\mapsto P_{J^c}$, where $J^c = \{1,2,\dots,b\} \setminus J$.
This isomorphism extends to an isomorphism 
between respective seed patterns in $R_{a,b}$ and~$R_{b-a,b}$. 
\end{remark}

\begin{lemma}\label{lem:Muir-embedding}
There is an injective ring homomorphism 
$R_{a-1,b-1} \to R_{a,b}$ which sends
$P_I$ to $P_{I \cup \{b\}}$.
\end{lemma}

\begin{proof}
The fact that the correspondence $P_I\mapsto P_{I \cup \{b\}}$ 
extends to a ring homomorphism follows from Muir's law (Proposition~\ref{prop:Muir}). 
\end{proof}

We call the map $R_{a-1,b-1} \to R_{a,b}$ described above the \emph{Muir embedding}.

Recall the notion of a seed subpattern from 
Definition~\ref{def:clustersubalgebra}. 
%(see also
%Remark~\ref{rem:clustersubalgebra}).

\begin{lemma}
\label{lem:subpattern}
The Muir embedding sends the seed pattern in $R_{a-1,b-1}$ defined by $\Sigma_{a-1,b-1}$ 
to a subpattern of the seed pattern in $R_{a,b}$ defined by~$\Sigma_{a,b}$.
\end{lemma}

\begin{proof}
Delete the bottom row of vertices in the rectangles quiver~$Q_{a,b}$. 
Freeze the vertices of the new bottom row of the resulting quiver. 
Delete any arrows connecting two frozen vertices.  
Then remove the index~$b$ from every label. 
This will produce the rectangles quiver $Q_{a-1,b-1}$, with its standard labeling. 
%where corresponding vertices of $Q_{a-1,b-1}$
%and $Q_{a,b}$ are labeled $I$ and $I \cup \{b\}$, respectively.
\end{proof}

\begin{remark}
Similarly,  the seed pattern  in $R_{a,b-1}$ defined by $\Sigma_{a,b-1}$
is isomorphic to a seed subpattern of the seed pattern in $R_{a,b}$ defined by~$\Sigma_{a,b}$. 
(This involves deleting the rightmost column of the quiver~$Q_{a,b}$.) 
\end{remark}

\begin{lemma}
\label{lem:PluckerInCluster}
$R_{a,b} \subset \mathcal{A}(\Sigma_{a,b})$. 
\end{lemma}

\begin{proof}
Since the Pl\"ucker coordinates $P_J$ generate the Pl\"ucker ring~$R_{a,b}$,
it suffices to show that each $P_J$ lies in the cluster algebra
$\mathcal{A}(\Sigma_{a,b})$.

We will prove this claim by induction on $a$. 
The base case $a=2$ holds from our earlier analysis of $\Gr_{2,b}$.
By induction, for any $(a-1)$-element
subset $I$ of $\{1,2,\dots, b-1\}$, there is a sequence of mutations
that we can apply to $\Sigma_{a-1,b-1}$ to obtain the Pl\"ucker coordinate
$P_I$.  By Lemma \ref{lem:subpattern},
we can apply the same sequence of mutations to 
$\Sigma_{a,b}$ to obtain the Pl\"ucker coordinate $P_{I \cup \{b\}}$.
Consequently all Pl\"ucker coordinates of the form
$P_{I \cup \{b\}}$ belong to the cluster algebra~$\mathcal{A}(\Sigma_{a,b})$.  
It then follows by Exercise~\ref{ex:cyclic} 
that \emph{all} Pl\"ucker coordinates in $R_{a,b}$ lie in~$\mathcal{A}(\Sigma_{a,b})$.  
\end{proof}

\begin{proof}[Proof of Theorem \ref{th:Grassmain}]
In view of \cref{lem:PluckerInCluster}, it remains to show that 
$\mathcal{A}(\Sigma_{a,b})\subset R_{a,b}$.
By the Starfish lemma, all we need to establish is that 
mutating at any vertex $P_J$ of the quiver $Q_{a,b}$ 
yields a cluster variable $(P_J)'$ which is coprime to~$P_J$.

The mutable vertices of $Q_{a,b}$ all have degree $4$ or~$6$.
If we mutate at a degree $4$ vertex of $Q_{a,b}$, then the
corresponding exchange relation is a $3$-term Grassmann-Pl\"ucker relation,
and the resulting cluster variable is a Pl\"ucker coordinate.
Since the determinant is an irreducible polynomial, the old and new cluster variables are coprime.

If we mutate at a degree $6$ vertex, the resulting cluster 
variable is not a Pl\"ucker coordinate; however 
one can use an argument similar to that from Section
\ref{sec:rings-baseaffine} to prove that the old and new
cluster variables are still coprime.  
In this case, our degree $6$ vertex is labeled by 
some Pl\"ucker coordinate $P_{ijkS}$, 
where the subset $S\subset\{1,\dots,b\}$ of size $a-3$ is disjoint from $\{i,j,k\}$. 
The exchange relation has the form 
$$P_{ijkS} P'_{ijkS} = P_{ikfS} P_{ijdS} P_{jkeS} + P_{ikdS} P_{ijeS} P_{jkfS},$$
where the subset $\{d,e,f\}\subset\{1,\dots,b\}$ is disjoint from $\{i,j,k\} \cup S$.
One can then check that 
$P'_{ijkS} = P_{ikfS} P_{jdeS} - P_{jkdS} P_{iefS}$. 

We need to show that $P_{ijkS}$ and $P'_{ijkS}$ are coprime.  
Since the determinant is an irreducible polynomial, 
the only way $P_{ijkS}$ and $P'_{ijkS}$ can fail to be coprime is if 
$P_{ijkS}$ divides~$P'_{ijkS}$.  Let us show that this cannot happen.
Let $z$ be a generic $3 \times b$ matrix; let us augment it to an 
$a \times b$ matrix $\hat{z}$ by adding new rows $4$ through $a$,
where the submatrix located in rows $4\dots a$ and columns $S$
is the identity, and all other entries in rows $4\dots a$ are $0$.
	Then $P_{ijkS}(\hat{z})$ divides~$P'_{ijkS}(\hat{z})$ implies
	that $P_{ijk}(z)$ divides~$P'_{ijk}(z)$. 
%By Muir's
%law, we may assume that $S=\varnothing$.
%Both $P_{ijk}$ and $P'_{ijk}$ are polynomials in the entries of a 
%generic $3 \times b$ matrix~$z$.  
If we specialize $z_{1d}=z_{1e}=z_{2d}=z_{2e}=0$, 
%set the top two entries in each of the columns $d$ and $e$ of $z$ to be~$0$.
\linebreak[3]
then  $P_{ijk}$ is unchanged whereas 
$P_{jde}$ becomes~$0$, $P_{jkd}$ becomes~$z_{3,d} \, \Delta_{12,jk}(z)$,
and $P_{ief}$ becomes~$-z_{3,e}\, \Delta_{12,if}$.
Thus $P'_{ijk}$ specializes to $z_{3,d} \, \Delta_{12,jk}(z) \,z_{3,e} \, \Delta_{12,if}$.
Now if $P_{ijk}$ divides $P'_{ijk}$ then the same is true after specialization.
But $P_{ijk}$ does not divide
$z_{3,d} \, \Delta_{12,jk}(z)\, z_{3,e} \, \Delta_{12,if}$.  
Thus $P_{ijkS}$ and $P'_{ijkS}$ are coprime, and we are done. 
\end{proof}

Typically, the cluster structure in a  Pl\"ucker ring $R_{a,b}$ is of infinite type.  
The few exceptional cases where it has finite
type are listed in Table~\ref{tab:clustertype2}.

\begin{table}[ht]
\begin{center}
\begin{tabular}{| l | l |}
\hline
Ring & Cluster type \\
\hline 
\hline
& \\[-.4cm]
$R_{2,b}$ and $R_{b-2,b}$ & $A_{b-3}$ \\[.1cm] \hline
& \\[-.4cm]
$R_{3,6}$ & $D_4$ \\[.1cm] \hline
& \\[-.4cm]
$R_{3,7}$ and $R_{4,7}$ & $E_6$ \\[.1cm] \hline
& \\[-.4cm]
$R_{3,8}$ and $R_{5,8}$  & $E_8$ \\[.1cm] \hline
& \\[-.4cm]
$R_{a,b}$ for other $a$, $b$  & infinite type \\ \hline
\end{tabular}
\vspace{.2cm}
\end{center}
\caption{The type of the cluster structure of $R_{a,b}$.}
\label{tab:clustertype2}
\vspace{-.2cm}
\end{table}

\begin{remark}
\label{rem:non-plucker}
It is natural to seek an explicit description for all cluster and coefficient variables 
for the cluster structure in~$R_{a,b}$ described above. 
As we have seen, this set contains all Pl\"ucker coordinates. 
In the cases $a=2$ and $a=b-2$, there is nothing else;
in all other cases, the list includes non-Pl\"ucker cluster variables. 
For the finite types listed in \cref{tab:clustertype2},
the formulas for non-Pl\"ucker variables were given in~\cite{scott}. 
Beyond finite type, the problem remains open. 
The case $a=3$ was extensively studied~in~\cite{fomin-pylyavskyy}. 
\end{remark}

\begin{remark}
The above construction  can be adapted to yield a cluster structure 
in the coordinate ring $\CC[\operatorname{Mat}_{a \times (b-a)}]$
of the affine space of $a\times (b-a)$ matrices.
Append an identity matrix to the right of an \hbox{$a \times (b-a)$} matrix~$z$
to obtain an $a \times b$ matrix~$z'$. 
Up to $\SL_a$ action, the only restriction on~$z'$ is that its 
$a\times a$ minor occupying the last $a$ columns is equal~to~1.  
We can now identify the minors of~$z$ 
with the maximal minors of~$z'$ (the Pl\"ucker coordinates):
%for an $a$-element subset $J\subset\{1,2,\dots, b\}$,
a Pl\"ucker coordinate $P_J\in R_{a,b}$ 
corresponds to the minor $\varphi(P_J) = \Delta_{KL}(z)\in \CC[\operatorname{Mat}_{a \times (b-a)}]$ 
whose row and column sets are given by 
\begin{align*}
K&=(\{b-a+1,b-a+2,\dots,b\}\setminus J)-b+a, \\
L&=J \cap \{1,2,\dots, b-a\};
\end{align*}
here the notation $S-c$ means $\{s-c\mid  s\in S\}$.
Given a seed for $R_{a,b}$, applying the map $\varphi$ to all cluster and coefficient variables
(except for the coefficient variable $P_{b-a+1,b-a+2,\dots,b}$) 
yields a seed for $\CC[\operatorname{Mat}_{a \times b}]$.

\nopagebreak

In the special case $b=2a$, this identification shows that 
the cluster structures in the rings
$\CC[\operatorname{Mat}_{a \times a}]$ and $\CC[\SL_a]$
introduced in Section~\ref{sec:rings-matrices} 
are very closely related to the cluster structure in the Pl\"ucker ring~$R_{a,2a}$.
\end{remark}

%\begin{remark}
%There is a close relation between 
%the cluster structures that we have discussed 
%on $\CC[\SL_a]$ and
%on $\CGr{a}{2a}$.  More specifically, by appending
%an identity matrix to an $a \times a$ matrix,
%we can identify the set of all minors of an $a\times a$ matrix
%with the set of maximal minors of an $a \times 2a$ matrix.  
%Given $J$ an $a$-element subset of $\{1,2,\dots, 2a\}$,
%this identifies the Pl\"ucker coordinate $P_J$ with the 
%minor $\phi(P_J) = \Delta_{K,L}$, where 
%the row set $K = \{\{a+1,a+2,\dots, 2a\}\setminus J\}-a$,
%the column set $L = J \cap \{1,2,\dots, a\}$,
%and the notation $S-a$ means 
%$\{s_i-a\ \vert \ s_i \in S\}$.

%Given a seed for $\CC[\SL_a]$, if we simply apply 
%the map $\phi$ to all cluster variables, we obtain 
%a seed for $\CGr{a}{2a}$.  
%See notes DoubleWiringDiagram-forSergey.pdf.
%\end{remark}

%\newpage

The constructions presented in Sections~\ref{sec:rings-baseaffine}--\ref{sec:plucker-rings}
can be generalized and modified to build cluster structures in many other rings 
naturally arising in the context of classical invariant theory as well as algebraic Lie theory. 
We already mentioned generalizations and extensions to other semisimple Lie groups,
their subgroups, parabolic quotients, and double Bruhat cells;
the excellent survey~\cite{gls-survey-Lie} describes the state of the art circa~2013. 

To keep  our exposition within reasonable bounds, 
we did not discuss the constructions of cluster structures in the rings of $\SL_k$ invariants 
of collections of vectors, covectors, 
and/or matrices \cite{carde, fomin-pylyavskyy, fomin-pylyavskyy-pnas}. 
Likewise, we left out the treatment of the Fock-Goncharov configuration spaces \cite{fock-goncharov-ihes,fock-goncharov-dual} and related topics of higher Teichm\"uller theory. 

%$Conf(Fl)$ $\operatorname{SL}_k$ local systems on Riemann surfaces
%  (Fock-Goncharov varieties). Invariants of collections of ``affine flags.''

%\newpage

\section{Defining cluster algebras by generators and relations}
\label{sec:generators+relations}

One traditional way of describing a commutative algebra~$\AA$ (say over~$\CC$) 
is in terms of generators and relations. 
In this approach, $\AA$~is represented as a quotient of a $\CC$-algebra
$\CC[\zz]=\CC[z_1,z_2,\dots]$ 
freely generated by a (finite or countable) set of ``variables'' $\zz=\{z_1,z_2,\dots\}$ 
modulo an explicitly given ideal~$I\subset \CC[\zz]$. 
In typical applications, the set~$\zz$ is finite (so that $P=\CC[\zz]$ is a polynomial ring) 
and the ideal $I$ is finitely generated: $I=\langle g_1,\dots,g_N\rangle$,
where $g_1,\dots,g_N$ are polynomials in the variables~$z_1,z_2,\dots$. 
(By common abuses of terminology, we identify polynomials $f\in\CC[\zz]$ 
with the elements of $\AA\cong\CC[\zz]/I$ they represent. 
We also conflate the polynomials $g\in I$ 
with the relations $g(z_1,z_2,\dots)=0$ holding in~$\AA$.) 

The definition of a cluster algebra (Definition~\ref{def:cluster-algebra}) is set up differently:  
a~cluster algebra~$\AA$ is defined inside a field~$\mathcal{F}$  
of rational functions in several variables 
as the algebra generated by certain (recursively determined) elements of~$\mathcal{F}$,
the cluster variables of~$\AA$. 
While the relations among these generators are not given explicitly, 
we do know some of them, namely the exchange relations~\eqref{eq:exch-rel-geom}. 

It is natural to try to extract from this definition 
a traditional-style description of a cluster algebra as a quotient of a polynomial ring.
This runs into two issues. 
First, the set of cluster variables is typically infinite. 
Second, the exchange relations do not, in general, generate the ideal of~all relations
among cluster variables. 
We will discuss these two issues one by~one. 

The following statement, provided here without  proof, shows that some cluster algebras
are \emph{not} finitely generated: 

\begin{proposition}[{\cite[Theorem~1.26]{ca3}}] 
\label{pr:ca3-finite-gen}
A cluster algebra of rank~$3$ with trivial coefficients  
is finitely generated if and only if it has an acyclic seed. 
\end{proposition}

In the terminology of Example~\ref{ex:mutation-acyclic}, 
a cluster algebra defined by a 3-vertex quiver~$Q$ with no frozen vertices 
is finitely generated if and only if $Q$ is mutation-acyclic. 

To illustrate Proposition~\ref{pr:ca3-finite-gen}, the cluster algebra defined by the Markov quiver 
(see Figure~\ref{fig:markov-quiver}) is not finitely generated. 
%Indeed, its mutation class consists of a single quiver that is not acyclic. 
The following result, combined with  Proposition~\ref{pr:ca3-finite-gen}, 
provides many more examples. 

\begin{proposition}[{\cite[Theorem~1.2]{Beineke-Brustle-Hille}}]
\label{pr:BBH}
Let $Q(a,b,c)$ denote the quiver with
vertices $1, 2, 3$ and $a+b+c$ arrows: 
$a$ arrows $1\to 2$,
$b$ arrows $2\to 3$, and 
$c$ arrows $3\to 1$. 
(See Figure~\ref{fig:rank3-cyclic}.) The following are equivalent:
\begin{itemize}[leftmargin=.2in]
\item 
the quiver $Q(a,b,c)$ is not mutation-acyclic; 
\item 
$a,b,c\ge 2$ and 
$\det\left(\begin{smallmatrix} 2 & a & c \\
a & 2 & b\\
c & b & 2
\end{smallmatrix}\right)\ge 0$. 
\end{itemize}
 
\end{proposition}

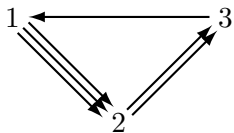
\begin{figure}[ht]
\setlength{\unitlength}{2pt} 
\begin{picture}(40,25)(0,0) 
\thicklines 
\put(2,18){\vector(1,-1){16}} 
\put(3,19){\vector(1,-1){16}} 
\put(1,17){\vector(1,-1){16}} 
%\put(22,2){\vector(1,1){16}} 
\put(21.5,2.5){\vector(1,1){16}} 
\put(22.5,1.5){\vector(1,1){16}} 
\put(37,20){\vector(-1,0){34}} 

\put(0,20){\makebox(0,0){$1$}}
\put(20,0){\makebox(0,0){$2$}}
\put(40,20){\makebox(0,0){$3$}}

\end{picture} 

\caption{The 3-vertex quiver $Q(a,b,c)$ with $a=3$, $b=2$, $c=1$.}
\label{fig:rank3-cyclic}
\end{figure}

Various cluster algebras arising in Lie theory,
such as the ones discussed in 
Sections~\ref{sec:rings-baseaffine}, \ref{sec:rings-matrices}, and \ref{sec:plucker-rings},
are finitely generated, for the reasons given in 
Theorems~\ref{th:hilbert}, \ref{th:hilbert-dolgachev}, and~\ref{th:A^U-fin-generated}. 

Another class of finitely generated cluster algebras is provided by the following result,
stated here without  proof. 

\begin{theorem}[{\cite[Corollary~1.21]{ca3}}]
\label{th:acyclic-finite-gen}
Any cluster algebra defined by an acyclic quiver
(with no frozen vertices) is finitely generated. 
In fact, it is generated by the cluster variables belonging to the initial acyclic seed 
together with the cluster variables obtained from this seed by a single mutation;
the ideal of relations among these cluster variables 
is generated by %the left-hand sides of 
the exchange relations %corresponding to mutations 
out of the initial acyclic seed. 
\end{theorem}

Theorem~\ref{th:acyclic-finite-gen} was extended in~\cite{Muller-locally-acyclic} 
to the much larger class of ``locally acyclic'' cluster algebras. 

\begin{remark}
Any cluster algebra of finite type is finitely generated. 
This follows from the appropriate generalization of Theorem~\ref{th:acyclic-finite-gen}:
a cluster algebra of finite type always has an acyclic seed of the kind described in
Theorem~\ref{th:finite-type-classification}. 
(In the quiver case, the quiver  at such a seed is an orientation of the corresponding Dynkin diagram.) See \cite[Remark~1.22]{ca3}. 

Many cluster algebras of infinite type (and even infinite mutation type) 
are finitely generated. For example, any Pl\"ucker ring $R_{a,b}$ is finitely generated
whereas its cluster structure is typically of infinite type, see Table~\ref{tab:clustertype2}. 
\end{remark}

When thinking about finite generation, it is helpful to keep in mind 
that by \cref{pr:fin-gen-by-cl-var}, 
a cluster algebra~$\AA$ is finitely generated if and only if $\AA$ is generated by a finite subset of cluster and coefficient variables. 

\iffalse
\begin{proof}
Suppose $\AA$ is generated by a finite subset $\zz\subset\AA$.
By the definition of a cluster algebra, each element of $\zz$ can be written as a polynomial
in (a finite subset of) cluster and coefficient variables. 
Therefore $\AA$ is generated by the union of these finite subsets. 
\end{proof}
\fi

We next turn to the problem of describing the ideal of relations satisfied by 
a set of generators of a cluster algebra. 
Even when a cluster algebra is of finite type, this is a delicate issue,
as Examples~\ref{ex:A_T-B2} and~\ref{ex:Gr(3,6)} below demonstrate. 

\pagebreak[3]

\begin{example}
\label{ex:A_T-B2}
Let $\AA=\AA(1,2)$ be the cluster algebra of type~$B_2$ with trivial coefficients. 
In the notation of Example~\ref{example:A(1,2)}, 
$\AA$ is generated by the 6-periodic sequence of cluster variables 
$z_1,z_2,\dots$ satisfying the exchange relations 
\begin{align}
\label{eq:z1z3-B2}
z_1z_3 - z_2^2-1=0,\\
\label{eq:z2z4-B2}
z_2z_4 - z_3-1=0, \\
\label{eq:z3z5-B2}
z_3z_5 - z_4^2-1=0,\\
\label{eq:z4z6-B2}
z_4z_6 - z_5-1=0,\\
\label{eq:z5z1-B2}
z_5z_1 - z_6^2-1=0,\\
\label{eq:z6z2-B2}
z_6z_2 - z_1-1=0. 
\end{align}
It turns out that these relations do \emph{not} generate the ideal of all relations satisfied by $z_1,\dots,z_6$. 
It is not hard to check (using the formulas in Example~\ref{example:A(1,2)}) that 
these cluster variables also satisfy the relations
\begin{align}
\label{eq:z1z4-B2}
z_1z_4-z_2-z_6=0, \\ 
\label{eq:z3z6-B2}
z_3z_6-z_4-z_2=0, \\ 
\label{eq:z2z5-B2}
z_5z_2-z_6-z_4=0, 
\end{align}
none of which lies in the ideal generated by \eqref{eq:z1z3-B2}--\eqref{eq:z6z2-B2}
inside the polynomial ring in six formal variables $z_1,\dots,z_6$. 
\end{example}

The last claim, like several others made below in this section, 
can be readily checked using any of the widely available software packages 
for commutative algebra.

\begin{example}
\label{ex:Gr(3,6)}
Consider the Pl\"ucker ring $R_{3,6}=\CC[\widehat\Gr(3,6)]$,
viewed as a cluster algebra of finite type~$D_4$,
as explained in Section~\ref{sec:plucker-rings}. 
The set of its cluster variables contains the Pl\"ucker coordinates~$P_{ijk}$. 
This cluster algebra is graded, with $\deg(P_{ijk})=1$ for all $i,j,k$. 
All cluster variables are homogeneous elements, 
and all relations among them are homogeneous as well. 
These relations in particular include the Grassmann-Pl\"ucker relation
\begin{equation}
\label{eq:Plucker36}
P_{135} P_{246}-P_{134}P_{256}-P_{136}P_{245}-P_{123}P_{456} = 0.
\end{equation}
The relation \eqref{eq:Plucker36} cannot be written as a polynomial 
combination of exchange relations, since all those relations have degree at least~2, 
and none of them involves the monomial $P_{135} P_{246}$. 
%(or the monomial $P_{135}$ or the monomial $P_{246}$). 
Thus the ideal of relations among the cluster variables of~$R_{3,6}$ 
is \emph{not} generated by the exchange relations. 

This example can be extended to bigger Grassmannians using Muir's Law (Proposition~\ref{prop:Muir}).
\end{example}

On the bright side, the rings discussed in Examples~\ref{ex:A_T-B2} and~\ref{ex:Gr(3,6)} 
do have ``nice'' explicit presentations. 
The ideal of relations among the Pl\"ucker coordinates is generated by the classical (quadratic) Grassmann-Pl\"ucker relations. 
As to the cluster algebra $\AA(1,2)$ from Example~\ref{ex:A_T-B2}, 
the ideal of relations among the six cluster variables $z_1,\dots,z_6$ 
is generated by (the left-hand sides of) 
the relations \eqref{eq:z1z3-B2}--\eqref{eq:z6z2-B2} and \eqref{eq:z1z4-B2}--\eqref{eq:z2z5-B2}. 

In \cref{thm:clusterradical} below,
we will provide a general description of 
a finitely generated cluster algebra~$\AA$ 
in terms of generators and relations. 
This will require some preparation. 

\begin{definition}
\label{def:T-subalgebra}
Recall that a cluster algebra~$\AA$ of rank~$n$ is defined by a seed pattern 
whose seeds are labeled by the vertices of the $n$-regular tree~$\TT_n$,
cf.\ Definition~\ref{def:Tn}.
Let $T$ be a finite subtree of~$\TT_n$. 
For $i\in\{1,\dots,n\}$, let $T[i]$ denote the forest obtained from~$T$
by removing the edges labeled by~$i$. 
We denote by $\zz_T$ the (finite) set of formal variables which includes 
\begin{itemize}[leftmargin=.2in]%[leftmargin=.2in]
\item 
one formal variable for each coefficient variable of~$\AA$, and
\item
	one formal variable for each connected component of $T[i]$, 
for every $i\in\{1,\dots,n\}$. 
\end{itemize}
The formal variable % $z\in\zz_T$
associated to a connected component $C$ of $T[i]$ 
naturally corresponds to the unique cluster variable in~$\AA$
that is indexed by~$i$ within each of the labeled seeds in~$C$.

We denote by $\CC[\zz_T]$ the ring of polynomials in the set of variables~$\zz_T$. 
The \emph{exchange ideal} $I_T\subset \CC[\zz_T]$ 
is the ideal generated by the exchange relations corresponding to the edges of~$T$.
More precisely, for each exchange relation $zz'=M_1+M_2$
%(cf.~\eqref{eq:exch-rel-geom}) 
corresponding to an edge of $T$ (here $z,z'\in\zz_T$, and $M_1,M_2$ are monomials
in the elements of~$\zz_T$), 
the exchange ideal $I_T$ contains the polynomial $zz'-M_1-M_2$. 

Let $\mathcal{Z}_T$ denote the set of all cluster and coefficient variables appearing in the seeds 
labeled by the vertices of~$T$. 
Let $\AA_T\subset\AA$ be the subalgebra generated by~$\mathcal{Z}_T$. 
We are especially interested in the cases where $\AA_T=\AA$, 
so that $\mathcal{Z}_T$~generates the entire cluster algebra~$\AA$. 

In what follows, we habitually use the same notation for a formal variable $z\in\zz_T$
and the corresponding cluster variable $z\in\mathcal{Z}_T$. 
When this abuse of notation becomes dangerously confusing,
we write $f(\zz_T)$ and $f(\mathcal{Z}_T)$ to distinguish between a polynomial 
$f\in\CC[\zz_T]$ and its evaluation in~$\AA_T\subset\AA$. 
\end{definition}

\begin{remark}
\label{rem:fin-gen-by-tree}
Let $\AA$ be a finitely generated cluster algebra. 
By Proposition~\ref{pr:fin-gen-by-cl-var}, $\AA$ is generated
by a finite subset~$\zz$ of cluster and coefficient variables. 
Enlarging $\zz$ if necessary, we may furthermore assume that all cluster variables in~$\zz$
come from clusters connected to each other by mutations that pass through 
clusters all of whose cluster variables belong~to~$\zz$. 
It follows that we can find a finite tree~$T$ such that $\AA_T=\AA$. 
\end{remark}

%We illustrate the notions introduced in Definition~\ref{def:T-subalgebra} by the following example. 

\begin{example}
\label{ex:A_T-B2-continued}
Let $\AA=\AA(1,2)$, as in Example~\ref{ex:A_T-B2}. 
We first consider the 3-vertex tree~$T$ corresponding to the following triple of clusters:
\begin{equation}
\label{eq:3-vertex-tree}
(z_1,z_2)\overunder{1}{}(z_3,z_2)\overunder{2}{}(z_3,z_4). 
\end{equation}
Then $\mathcal{Z}_T=\{z_1,z_2,z_3,z_4\}$,
in the notation of Definition~\ref{def:T-subalgebra}. 
The relations 
\begin{align}
z_6&=z_1z_4-z_2 \,,\\
\label{eq:z5=}
z_5&=z_4z_6-1=z_1z_4^2-z_2z_4-1=z_1z_4^2-z_3-2 
\end{align}
(cf.~\eqref{eq:z1z4-B2} and~\eqref{eq:z4z6-B2}) imply that 
the cluster algebra $\AA$ is generated by~$\mathcal{Z}_T$. 
This is an instance of Theorem~\ref{th:acyclic-finite-gen}:  
the four cluster variables in~$\mathcal{Z}_T$ come from the cluster $(z_2,z_3)$ and the two clusters
obtained from it by single mutations. 

The four elements of $\mathcal{Z}_T$
satisfy the exchange relations \eqref{eq:z1z3-B2} and~\eqref{eq:z2z4-B2}. 
The left-hand sides of these relations correspond to the generators 
of the exchange ideal~$I_T\subset\CC[\zz_T]$.  
One can check that this exchange ideal contains all relations satisfied by 
the cluster variables $z_1,z_2,z_3,z_4$, in agreement with Theorem~\ref{th:acyclic-finite-gen}. 
Consequently $\AA\cong \CC[\zz_T]/I_T$. 

Alternatively, consider the subtree $T$ spanning the four clusters
\begin{equation}
\label{eq:4-vertex-subtree}
(z_1,z_2)\overunder{1}{}(z_3,z_2)\overunder{2}{}(z_3,z_4)\overunder{1}{}(z_5,z_4). 
\end{equation}
Again, the set $\mathcal{Z}_T=\{z_1,\dots,z_5\}$ generates~$\AA$. 
The three exchange relations associated with the edges of~$T$ are 
\eqref{eq:z1z3-B2}, \eqref{eq:z2z4-B2}, and~\eqref{eq:z3z5-B2}. 
It turns out that the exchange ideal $I_T\subset\CC[\zz_T]$ generated by 
(the left-hand sides of) these relations does not contain 
some of the relations satisfied by the cluster variables $z_1,\dots,z_5$. 
For example, $f(\zz_T)=z_1z_4^2 -z_3-z_5-2\notin I_T$ 
%(here we treat the $z_i$ as formal variables), 
even though $f(\mathcal{Z}_T)=0$ in~$\AA$, cf.~\eqref{eq:z5=}. 
Thus for this choice of a tree~$T$, we have $\AA\not\cong \CC[\zz_T]/I_T$. 
(As the exchange ideal $I_T$ is radical in this instance, the gap cannot be explained by  
the discrepancy between the ideal of an affine variety and an ideal
coming from its set-theoretic description.) 
\end{example}

To describe the relationship between the algebra~$\AA_T$ 
and the quotient $\CC[\zz_T]/I_T$, we will need the following notation. 
Let $M_T$ denote the product of all formal variables in~$\zz_T$
that correspond to the (mutable) cluster variables.
We denote by 
\begin{equation}
\label{eq:sat-I_T}
J_T = (I_T:\langle M_T\rangle^\infty) = \{f\in\CC[\zz_T] : 
	%\exists a\  (M_T)^a f\in I_T\}
	  (M_T)^a f\in I_T \text{ for some }a\}
\end{equation}
the saturation of the exchange ideal $I_T$ by the principal ideal~$\langle M_T\rangle$.
In plain language, $J_T$ consists of all polynomials that can be multiplied
by a monomial so that the product lies in the exchange ideal~$I_T$. 

\begin{theorem}
\label{thm:clusterradical}
%Let $\AA$ be a cluster algebra of rank~$n$, 
%and let $T$ be a finite subtree of the $n$-regular tree~$\TT_n$. 
For a polynomial $f(\zz_T)$, the following are equivalent:
\begin{itemize}[leftmargin=.2in]
\item 
$f(\mathcal{Z}_T)\!=\!0$, i.e., $f$ describes a relation among cluster variables~in~$\AA_T$;  
\item 
$f(\zz_T)$ lies in the saturated ideal~$J_T$. 
\end{itemize}
We thus have the canonical isomorphism
\begin{equation}
\label{eq:clusterradical}
\AA_T\cong \CC[\zz_T]/J_T. 
\end{equation}
\end{theorem}

Informally, a polynomial in the cluster variables vanishes in the cluster algebra
if and only if this polynomial can~be multiplied by some monomial 
in cluster variables so that the product lies in the exchange ideal. 

\begin{remark}
%\label{rem:}
If the cluster algebra~$\AA$ is finitely generated, with $\AA=\AA_T$
(cf.\ \cref{rem:fin-gen-by-tree}), then \eqref{eq:clusterradical} provides an implicit presentation 
of~$\AA$ in terms of generators and relations. 
Furthermore, in each specific example, the saturated ideal $J_T$ can be explicitly computed 
using existing efficient algorithms of computational commutative algebra. 
\end{remark}

Before proving \cref{thm:clusterradical}, we illustrate it with a couple of examples. 

\begin{example}
\label{ex:4-vertex-subtree-B2}
Continuing with Example~\ref{ex:A_T-B2-continued}, 
let $\AA=\AA(1,2)$. 
Take the tree~$T$ shown in~\eqref{eq:4-vertex-subtree}. 
We saw that the polynomial $f=z_1z_4^2 -z_3-z_5-2$ does not lie in the exchange ideal~$I_T$,
even though $f$ describes an identity among the generators of~$\AA$.
On the other hand,
\begin{align*}
z_3f
&=z_3(z_1z_4^2 -z_3-z_5-2)\\
&=z_4^2(z_1z_3\!-\!z_2^2\!-\!1)\!+\!(z_2z_4\!+\!z_3\!+\!1)(z_2z_4\!-\!z_3\!-\!1)\!-\!(z_3z_5\!-\!z_4^2\!-\!1)
\in I_T, 
\end{align*}
so $f$ lies in the saturated ideal~$J_T$. 
\end{example}

\pagebreak[3]

\begin{example}
\label{ex:Gr36}
Continuing with Example~\ref{ex:Gr(3,6)}, 
consider the Pl\"ucker ring~$R_{3,6}$.
Although the Grassmann-Pl\"ucker relation \eqref{eq:Plucker36}
does not lie in the ideal generated by the exchange relations, 
we \emph{can} multiply \eqref{eq:Plucker36} by a monomial
(in fact, by a single variable) and get inside the ideal:
\begin{align}
\notag
       P_{124}(P_{135} P_{246}&-P_{134}P_{256}-P_{136}P_{245}-P_{123}P_{456}) \\
                &=P_{246} (P_{124}P_{135}-P_{123}P_{145}-P_{125}P_{134}) \label{3term1}\\
                &-P_{134} (P_{124}P_{256}-P_{125}P_{246}+P_{126}P_{245}) \label{3term2} \\
                &-P_{245} (P_{124}P_{136} - P_{123}P_{146}-P_{126}P_{134}) \label{3term3}\\
                &-P_{123} (P_{124}P_{456}-P_{145}P_{246}+P_{146}P_{245})\in I_T. \label{3term4}
        \end{align}
(Each of the four parenthetical expressions in \eqref{3term1}--\eqref{3term4} 
is a three-term Grassmann-Pl\"ucker relation, 
thus an instance of an exchange relation.) 
\end{example}

\begin{proof}[Proof of \cref{thm:clusterradical}]
Going in one direction, let us verify that if $f(\zz_T)\in J_T$, then 
$f(\mathcal{Z}_T)=0$. 
Suppose that $M(\zz_T)$ is a monomial such that $f(\zz_T)M(\zz_T)\in I_T$. 
Since every exchange relation holds when we substitute cluster variables into it, 
this implies that $f(\mathcal{Z}_T)M(\mathcal{Z}_T)=0$ in~$\AA_T$.
But $\AA_T$ is contained in a field~$\mathcal{F}$, and $M(\mathcal{Z}_T)$
is a nonzero element of~$\mathcal{F}$.
Therefore $f(\mathcal{Z}_T)=0$ as desired. 

To prove the converse, we will need the following definitions. 
Fix a root vertex $t_0$ in the tree~$T$. 
%We assume that $t_0$ is not a leaf. 
%(Such a choice can be made unless $T$ contains one or two vertices, in which case the claim is easy.) 
\iffalse
Given a variable $z\in\zz_T$,
let $d(z,t_0)$ denote the smallest distance within the tree~$T$ between vertex~$t_0$ 
and a vertex whose associated cluster contains~$z$ (or rather the corresponding cluster variable). 
For $M=M(\zz_T)$ a monomial,
we use any definition of distance $d(M,t_0)$ such that 
	\begin{itemize}[leftmargin=.2in]
		\item $d(M_1, t_0) < d(M_2, t_0)$ whenever 
			 $M_2$ contains a  variable that is further
			from $t_0$ than any  variable in~$M_1$;
		\item  $d(M_1, t_0) < d(M_2, t_0)$ whenever
			both $M_1$ and $M_2$ contain variables at distance
			$d$ from $t_0$, and $d$ is the 
			largest distance that occurs for either $M_1$ or~$M_2$, and moreover 
			the number of variables at distance $d$
			in $M_1$ is less than the number
			of variables at distance $d$ in~$M_2$.
	\end{itemize}
For $f\in\CC[\zz_T]$,
we define $d(f,t_0)=\max_{M} \{d(M,\mathbf{x_0})\}$, where $M$ ranges over all monomials in $f$.
\fi
Let us linearly order the set $\zz_T$ so that
\begin{itemize}[leftmargin=.2in]
\item 
the coefficient variables and the variables associated with the root cluster
are smaller than the remaining variables; 
\item 
for each exchange relation $zz'=\cdots$, we have $z'<z$ if the ``cluster'' containing~$z'$
is closer to $t_0$ in~$T$ than the ``cluster'' containing~$z$. 
\end{itemize}
(We put the word ``cluster'' in quotation marks since we are dealing with formal variables 
rather than the associated cluster variables.) 

Let $f=f(\zz_T)$ be a polynomial such that $f(\mathcal{Z}_T)=0$.
We need to show that there exists a monomial $M\in\CC[\zz_T]$ such that $fM\in I_T$.
We will prove this by double induction: 
first, on the largest variable~$z\in\zz_T$ appearing in~$f$, 
and for a given~$z$, on the degree with which $z$ appears. 
In other words, we will argue as follows. 
Let $z\in\zz_T$ be the largest variable appearing in~$f$, say with degree~$d$.
Then we can assume, while proving the claim above, 
that a similar statement holds for any polynomial
that only involves the variables  smaller than~$z$, 
and perhaps also~$z$ in degrees~$<d$. 

We begin by writing
\[
f=zg+h, 
\]
where $g,h\in\CC[\zz_T]$ are polynomials, with $h$ not involving~$z$. 
Let 
\[
E=zz'-M_1-M_2 \in I_T 
\]
be the polynomial associated with the unique exchange relation 
among the variables in~$\zz_T$  that
corresponds to an edge in~$T$ and where $z'<z$. 
Thus $M_1,M_2\in\CC[\zz_T]$ are monomials 
that only involve variables smaller than~$z$.

Now set 
\[
f'=z'f-Eg.
\] 
Then $f'(\mathcal{Z}_T)=0$ because $f(\mathcal{Z}_T)=E(\mathcal{Z}_T)=0$.
Moreover the calculation 
\begin{align*}
f'&=z'f-Eg\\
&=zz'g+z'h-(zz'-M_1-M_2)g\\
&= z'h+(M_1+M_2)g 
\end{align*}
shows that $f'$ satisfies the conditions of the induction hypothesis. 
(Indeed, the polynomials $z'$, $h$, $M_1$ and $M_2$ only involve variables~$<z$ 
whereas $\deg_z(g)=\deg_z(f)-1$.) 
Hence there exists a monomial $M'=M'(\zz_T)$ satisfying $f'M'\in I_T$. 
Now let $M=M'z'$ and conclude that 
\[
fM = M'z'f=M'f'+M'Eg\in I_T, 
\]
proving the claim. 
\end{proof}

\backmatter

%%%%%%\include{references}
%\include{bibliography}

%bibliographystyle{amsalpha}
\bibliographystyle{acm}
\bibliography{bibliography}
\label{sec:biblio}

%RESTORE IN THE FINAL VERSION:
%\printindex

\end{document}